\documentclass[11pt]{article}
\usepackage[utf8]{inputenc}
\usepackage{amsthm}
\usepackage{amssymb}

\usepackage{enumerate}
\usepackage{cancel}
\usepackage{mathrsfs}
\usepackage{tikz}
\usepackage{graphicx}
\usepackage{gensymb}
\usetikzlibrary[shapes.geometric]
\usepackage{tabularx}
\usepackage{hyperref}
\usepackage{setspace}
\usepackage[top=1in, left=1in, right=1in, bottom=1in, footskip=0.5in]{geometry}
\usepackage{scrextend}
\usepackage{lipsum}
\usepackage{caption}
\usepackage{comment}
\usepackage{fancyhdr}
\usepackage{float}
\fancypagestyle{plain}{
\fancyhf{}
\fancyfoot[C]{\large\thepage}

}
\usepackage{xcolor}
\usepackage{pgffor}
\usepackage{graphbox}
\usepackage{soul}
\usepackage{color}
\usepackage{multirow}%
\usepackage{amsfonts}%
\usepackage[title]{appendix}%
\usepackage{textcomp}%
\usepackage{manyfoot}%
\usepackage{booktabs}%
\usepackage{algorithm}%
\usepackage{algorithmicx}%
\usepackage{algpseudocode}%
\usepackage{listings}%
\usepackage{subcaption}
\usepackage[T1]{fontenc}
\usepackage{array}
\usepackage[export]{adjustbox}

\usepackage{leftindex}

\usetikzlibrary{calc} 
\usetikzlibrary{external,quotes,chains}
\tikzset{vertex/.style={black,fill,draw,minimum size=6pt,inner sep=0pt,circle,thin},bold/.style={black,line width=0.6mm},plain/.style={black,thin},bold edges/.style=bold,plain edges/.style=plain,label distance=1mm,text node/.style={rectangle,fill=none,draw=none},every label/.style=text node,caption node/.style={text node,font=\Large}}

\newtheorem{theorem}{Theorem}[section]
    \newtheorem{lemma}[theorem]{Lemma}
    \newtheorem{corollary}[theorem]{Corollary}
    
    \newtheorem{observation}[theorem]{Observation}

    \newtheorem{property}[theorem]{Property}

    \newtheorem*{claim*}{Claim}
\theoremstyle{definition}
\newtheorem{definition}[theorem]{Definition}

\newcommand{\lv}[1]{\bar{\mathbf{#1}}} 
\newcommand{\wt}[1]{\mathbf{wt}_{\overrightarrow{#1}}} 
\newcommand{\imb}[1]{\mathbf{imb}_{\overrightarrow{#1}}} 

\newcommand{\ar}[1]{\overset{\rightarrow}{#1}}
\newcommand{\AR}[1]{\overrightarrow{#1}}

\title{New Results on Difference Distance Magic Labelings}
\author{Roza Aceska\thanks{Department of Mathematical Sciences, Ball State University, Muncie, IN, USA
(raceska@bsu.edu).}\and Niny Arcila-Maya\thanks{Department of Mathematics, San Francisco State University, San Francisco, CA, USA (niny.arcilamaya@sfsu.edu).}\and Joshua Carlson\thanks{Department of Mathematics and Computer Science, Drake University, Des Moines, IA, USA (joshua.carlson@drake.edu).}\and Alison Marr\thanks{Department of Mathematics and Computer Science, Southwestern University, Georgetown, TX, USA
(marra@southwestern.edu).}\and Miriam Parnes\thanks{Department of Mathematics, Towson University, Towson, MD, USA (mparnes@towson.edu).}\and Kathleen Ryan\thanks{Department of Mathematics and Computer Science, DeSales University, Center Valley, PA, USA (kathleen.ryan@desales.edu).}\and Houston Schuerger\thanks{Department of Mathematics, Tarleton State University, Stephenville, TX, USA
(hschuerger@tarleton.edu).}\and Jennifer F. Vasquez\thanks{Department of Mathematics, University of Scranton, Scranton, PA, USA (jennifer.vasquez@scranton.edu).}}

\begin{document}

\maketitle

\begin{abstract}A graph labeling assigns values to the components of a graph (vertices, edges, etc.).  
In particular, distance magic labelings have been widely studied in undirected graphs. In such a labeling, the vertices are labeled with unique values from one up to the number of vertices so that the sum of labels on the neighbors of any vertex is the same across all vertices. For oriented graphs, a related concept of distance difference magic has been studied. In a distance difference magic labeling, each vertex is given a unique value from one up to the number of vertices such that for each vertex the sums of the labels of vertices in the in-neighborhood minus the sums of the labels of vertices in the out-neighborhood equals zero.  In this paper, we expand on this concept by showing a connected difference distance magic oriented graph on $n$ vertices exists for each integer $n \geq 5$. We also construct arbitrarily large difference distance magic oriented graphs from smaller ones using a new graph sum and exhibit a connection between linear algebra and this type of labeling.
\end{abstract}

\noindent {\bf Keywords:} graph labeling, oriented graph, difference distance magic, weighted graph sum

\section{Introduction}

A graph labeling assigns values to the components of a graph (vertices, edges, etc.). Often we aim to assign these values so that certain conditions hold. There are a variety of different named graph labelings and Gallian's survey is one source of information to learn about what has been done in the field \cite{Gal}.  In particular, distance magic labelings have been widely studied in undirected graphs. Here the vertices are labeled with unique values from one up to the number of vertices so
that the sum of labels on the neighbors of any vertex is the same across all vertices. In \cite{FM}, difference distance magic labelings are defined for oriented graphs. In such a labeling, the vertices are labeled with unique values from one up to the number of vertices so that for each vertex the sum of the labels of vertices in the in-neighborhood minus the sum of the labels of vertices in the out-neighborhood equals zero.

Consider an oriented graph $\AR{G}$ on  $n$ vertices. Let $f$ be a bijective labeling of the vertices of $\AR{G}$. If at every vertex $v$ of $\AR{G}$, the difference of the sum of the labels from the in-neighborhood of $v$ minus that of the out-neighborhood of $v$ is equal to zero (i.e., if the sum of the labels from the in-neighborhood is equal to the sum of the labels of the out-neighborhood), then we say that $f$ is a difference distance magic (DDM) labeling of $\AR{G}$. If a graph $G$ can be oriented so that a DDM labeling exists, then we say that $G$ is difference distance magic orientable (DDMO).

In this paper, we expand on this work.  In Section \ref{def}, we include the definitions of common graph theory terminology for the reader's reference.  In Section \ref{Imbalance}, we introduce the concept of imbalance which plays a key role in being able to construct larger DDMOG from smaller ones, exhibit the connection between a graph's imbalance and properties of its adjacency matrices, and identify a process for taking a DDMOG of imbalance 1 and from it producing a graph with imbalance 0 by adding a single additional vertex.  In Section \ref{arbitrarily_large}, we consider a vertex gluing technique known  as vertex coalescence \cite{coal} and show that by utilizing this technique, for each positive integer $n \geq 5$ one can construct a connected DDMOG on $n$ vertices.  In Section \ref{operations}, we consider methods for constructing larger disconnected DDMOGs. We also introduce a new  weighted graph sum  as a tool for yielding connected DDMOGs from two smaller oriented graphs.


\section{Definitions}\label{def}

A \emph{(simple) graph} $G$ is defined to be a pair $(V, E)$, consisting of a set of \emph{vertices} $V=V(G)$ and a set $E=E(G)$ of pairs of distinct vertices, called \emph{edges}.  If $u, v \in V(G)$ and $uv \in E(G)$, then we say $u$ and $v$ are \emph{neighbors} (or are \emph{adjacent}) and that the edge $uv$ is \emph{incident} to $u$ and $v$ (and is illustrated in diagrams by $u \bullet$---$\bullet v$).  The \emph{neighborhood} of a vertex $v \in V(G)$ is the set of the neighbors of $v$ and is denoted $N_G(v)$. The \emph{closed neighborhood} of a vertex $v$ is the set $N_G(v) \cup \{v\}$ and is denoted $N_G[v]$.  The \emph{degree} of $v$ is $|N_G(v)|$. In either notation, if the graph $G$ is clear, the subscript is not included.

A \emph{digraph (directed graph)} $D$ is defined to be a pair $(V,E)$, consisting of a set of \emph{vertices} $V(D)$ and a collection of ordered pairs of vertices $E(D)$, called \emph{edges}. If $u, v \in V(D)$ and $(u,v) \in E(D)$, then we say $u$ and $v$ are  \emph{neighbors} (or are \emph{adjacent}) and that the edge $(u,v)$ is \emph{incident} to $u$ and $v$, and is illustrated in diagrams by $u\bullet$ \hspace{-0.1in} $\longrightarrow$ \hspace{-0.1in} $\bullet v$. When clarity is not at stake, we at times denote the directed edge $(u, v)$ more simply as $uv$. Note that because order is important in digraphs, the edge $uv$ is different from the edge $vu$. The \emph{in-neighborhood} of a vertex $v$ is the set of vertices $u \in V(D)$ such that $(u,v) \in E(D)$ and is denoted $N^+(v)$.  The \emph{out-neighborhood} of a vertex $v$ is the set of vertices $u \in V(D)$ such that $(v,u) \in E(D)$ and is denoted $N^-(v)$. When necessary to avoid confusion, we use $N_D^+(v)$ and $N_D^-(v)$ to denote the in- and out-neighborhoods of $v$, respectively, in a specified digraph $D$.

An \emph{oriented graph} $\AR{G}$ is a digraph such that $uv \in E(\AR{G})$ implies $vu \not \in E(\AR{G})$. The \emph{underlying graph} $G$ of an oriented graph $\AR{G}$ is the graph that results when the orientations of the edges are removed from $\AR{G}$. If the underlying graph of an oriented graph $\AR{G}$ is $G$, then we say that $\AR{G}$ is an \emph{orientation} of $G$. Notationally speaking, we reserve $\AR{G}$ for oriented graphs and $D$ for digraphs. 

A function $f: V(G) \rightarrow S$, for some set $S$, is called a \emph{vertex labeling}.  In this paper, every labeling we will consider is a vertex labeling, so for the sake of brevity, we refer to them simply as \emph{labelings}. 

\begin{definition}\label{defn:DDM}
 Let $\overrightarrow{G}$ be an oriented graph on $n$ vertices. Let $f: V(\overrightarrow{G}) \rightarrow \mathbb{Z}^+$ be a labeling of $\overrightarrow{G}$.  Then the \textbf{weight} of a vertex $v \in V(\overrightarrow{G})$, denoted $wt(v)$, is 
 $$wt(v):=\sum_{u \in N^+(v)} f(u)-\sum_{u \in N^-(v)} f(u).$$ If the labeling function we are working with is not clear, we use $wt_f(v)$ to denote the weight of the vertex $v$ with labeling $f$.
 Moreover, if $f$ is a bijection onto $\{1,2,\dots,n\}$ such that every vertex has weight 0, then $f$ is a \textbf{difference distance magic (DDM) labeling} of $\overrightarrow{G}$.\label{defn:vertexWeight}
\end{definition}

If an oriented graph $\AR{G}$ has a DDM labeling then $\AR{G}$ is a \textit{difference distance magic oriented graph (DDMOG)}. If there exists an orientation of a graph $G$ that is a DDMOG, $G$ is \textit{difference distance magic orientable (DDMO)}. Note that not every graph has an orientation that is difference distance magic.   In particular, by Property~\ref{property1}, any graph with minimum degree one or two is not DDMO. 
\begin{property}\label{property1}
 \cite{FM} If $\AR{G}$ is a connected, non-trivial DDMOG, then every vertex is incident to a minimum of
three arcs, not all directed into or away from the vertex. 
\end{property}
In addition, not every labeling of a graph that is DDMO is a DDM labeling. Similarly, not every orientation of a DDMO graph has a DDM labeling.  Figure \ref{fig:3versionsW4} shows a DDM labeling of an oriented wheel ($\overrightarrow{W_4}$), a labeling of $\overrightarrow{W_4}$ that is not a DDM labeling, and an orientation of $W_4$ that has no DDM labeling.

\begin{figure}[h]
    \centering
    \begin{subfigure}{0.29\textwidth}
        \centering
        \begin{tikzpicture}
        [scale=0.67, node distance=1.3cm,base node/.style={circle,draw,minimum size=21pt}]
            \node[base node, very thick, label={[text=red]above left:0}] (5){5};
            \node[base node, very thick, label={[text=red] above left:0}][left of=5] (3){3};
            \node[base node, very thick, label={[text=red]above:0}][above of=5] (1){1};
            \node[base node, very thick, label={[text=red]below:0}][below of=5] (4){4};
            \node[base node, very thick, label={[text=red]above left:0}][right of=5] (2){2};
            \path [->, very thick] (3) edge node[left] {} (5);
            \path [<-, very thick] (5) edge node[right] {} (2);
            \path [->, very thick](1) edge[bend right=50] node[left] {} (3);
            \path [->, very thick](4) edge[bend left=50] node[left] {} (3);
            \path [->, very thick](4) edge[bend right=50] node[left] {} (2);
            \path [->, very thick](1) edge[bend left=50] node[left] {} (2);
            \path [->, very thick] (5) edge node[left] {} (1);
            \path [->, very thick] (5) edge node[left] {} (4);
        \end{tikzpicture}
        \caption{A DDM labeling } \label{$W_4$ figure}
    \end{subfigure}
    \begin{subfigure}{0.3\textwidth}
        \centering
        \begin{tikzpicture}
        [scale=0.67, node distance=1.3cm,base node/.style={circle,draw,minimum size=21pt}]
            \node[base node, very thick, label={[text=red]above left:0}] (5){1};
            \node[base node, very thick, label={[text=red]above left:+6}][left of=5] (3){2};
            \node[base node, very thick, label={[text=red]above:-6}][above of=5] (1){3};
            \node[base node, very thick, label={[text=red]below:-6}][below of=5] (4){4};
            \node[base node, very thick, label={[text=red] above left:+6}][right of=5] (2){5};
            \path [->, very thick] (3) edge node[left] {} (5);
            \path [<-, very thick] (5) edge node[right] {} (2);
            \path [->, very thick](1) edge[bend right=50] node[left] {} (3);
            \path [->, very thick](4) edge[bend left=50] node[left] {} (3);
            \path [->, very thick](4) edge[bend right=50] node[left] {} (2);
            \path [->, very thick](1) edge[bend left=50] node[left] {} (2);
            \path [->, very thick] (5) edge node[left] {} (1);
            \path [->, very thick] (5) edge node[left] {} (4);
        \end{tikzpicture}
        \caption{A non-DDM labeling}  \label{fig:nonDDM W4}
    \end{subfigure}
     \begin{subfigure}{0.35\textwidth}
        \centering
        \begin{tikzpicture}
        [scale=0.67,node distance=1.3cm,base node/.style={circle,draw,minimum size=21pt}]
            \node[base node, very thick ] (5){};
            \node[base node, very thick][left of=5] (3){};
            \node[base node, very thick ][above of=5] (1){ };
            \node[base node, very thick ][below of=5] (4){};
            \node[base node, very thick ][right of=5] (2){};
            \path [->, very thick] (3) edge node[left] {} (5);
            \path [<-, very thick] (5) edge node[right] {} (2);
            \path [->, very thick](1) edge[bend right=50] node[left] {} (3);
            \path [->, very thick](4) edge[bend left=50] node[left] {} (3);
            \path [->, very thick](4) edge[bend right=50] node[left] {} (2);
            \path [->, very thick](1) edge[bend left=50] node[left] {} (2);
            \path [->, very thick] (1) edge node[left] {} (5);
            \path [->, very thick] (4) edge node[left] {} (5);
        \end{tikzpicture}

        \vspace{5.5mm}
        
        \caption{Has no DDM labeling}  \label{fig:nonDDM W4c}
    \end{subfigure}
    \caption{Examples of various orientations on $\ar{W_4}$, with labelings and vertex weights included in (a) and (b) } 
    \label{fig:3versionsW4}
\end{figure}
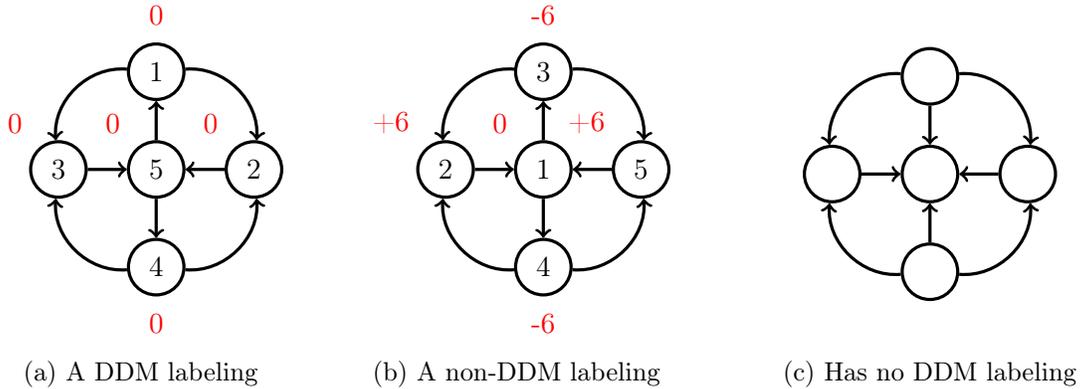

\section{Imbalances}\label{Imbalance}

Whereas the weight of a vertex $v$ indicates the difference in label sums  at $v$, the \textit{imbalance} of $v$ indicates the difference in the number of neighbors.

\begin{definition}
Given a digraph $D$, the \textbf{imbalance of a vertex $v$} is
\[imb(v) = |N^+(v)| - |N^-(v)|. \]
\end{definition}

\begin{definition}[Graph parameter]
The \textbf{imbalance of a digraph $D$} is
\[imb(D) = \max_{\forall v\in V(D)}  |imb(v)|.\]
\end{definition}

\begin{figure}[ht]
\begin{center}
\begin{tikzpicture}
[node distance=1.4cm, base node/.style={circle,draw,minimum size=15pt},scale=.8]

\node[base node, very thick, label={[text=red]left:-2}] (6) at (-0.8,2.5) {$v_6$};
\node[base node, very thick, label={[text=red]right:+1}] (5) at (1,2.5) {$v_5$};
\node[base node, very thick, label={[text=red]right:0}] (4) at (2.6,1) {$v_4$};
\node[base node, very thick, label={[text=red]right:0}] (3) at (2.5,-0.5) {$v_3$};
\node[base node, very thick, label={[text=red]right:+1}] (2) at (1,-1.8) {$v_2$};
\node[base node, very thick, label={[text=red]left:-1}] (1) at (-0.8,-1.8) {$v_1$};
\node[base node, very thick, label={[text=red]left:-1}] (8) at (-2.5,-0.5) {$v_8$};
\node[base node, very thick, label={[text=red]left:+2}] (7) at (-2.4,1) {$v_7$};


\path[->, very thick] (6) edge (4);
\path[->, very thick] (6) edge (2);
\path[->, very thick] (6) edge (1);
\path[->, very thick] (7) edge (6);

\path[->, very thick] (2) edge (7);
\path[->, very thick] (3) edge (7);
\path[->, very thick] (4) edge (7);
\path[->, very thick] (5) edge (7);

\path[->, very thick] (7) edge (8);

\path[->, very thick] (4) edge (5);
\path[->, very thick] (3) edge (5);

\path[->, very thick] (8) edge (3);
\path[->, very thick] (8) edge (4);

\path[->, very thick] (1) edge (2);
\path[->, very thick] (1) edge (4);

\path[->, very thick] (4) edge (3);

\end{tikzpicture}\caption{A digraph $D$ showing vertex imbalances. Note $imb(D) = 2$, while for instance $imb(v_3)=0$}\label{fig:VertexImbalance}
\end{center}
\end{figure}
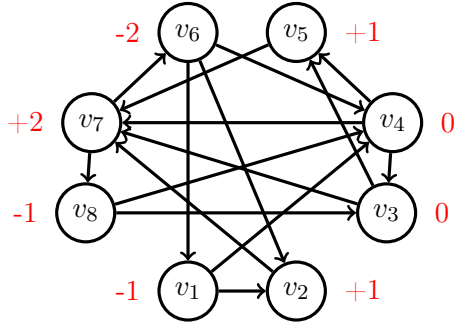

In this paper, we call a digraph {\it balanced} if every vertex has imbalance $0$, that is, if $imb(D)=0$.  See Figure \ref{fig:VertexImbalance} for an example of a digraph with imbalance 2.

Here we introduce some linear algebra terminology to better assist with the proofs related to imbalances.  We will use $\mathbf{0}$ and $\mathbf{1}$ for the zero column vector $[0, \dots, 0]^{T}$ and the column vector of all ones $[1, \dots, 1]^{T}$, each of length $n$, respectively. 

We follow the convention that the {\it adjacency matrix} $A=(a_{ij})_{1\leq i,j\leq n}$ of an oriented graph $\AR{G}$ with vertex set $V(\AR{G})=\{v_1, \dots, v_n\}$ is the  $n \times n$ matrix with entries $a_{ij}=1$ if $v_iv_j \in E(\overrightarrow{G})$ and $a_{ij} = 0$ otherwise.  In addition, we define the {\it skew-adjacency matrix} $S=(s_{ij})_{1\leq i,j\leq n}$ of $\AR{G}$ to be the matrix with entries given by 
\begin{align*}
    s_{ij} = 
    \begin{cases}
    -1, & \text{if } v_{i} v_{j} \in E(\overrightarrow{G}) \\
    1, & \text{if } v_{j}v_{i} \in E(\overrightarrow{G})  \\
    0, & \text{otherwise.}
    \end{cases}
\end{align*}

\noindent Note that $S=A^T-A$ is skew symmetric ($S=-S^T$) and that the sum of each row is the imbalance of the corresponding vertex of $\overrightarrow{G}$.

Given a labeling $f$ of $\overrightarrow{G}$,  let $\lv{x}=[f(v_1),\dots,f(v_n)]^{T}$ be the \textit{label vector}, and let the \textit{weight vector} and \textit{imbalance vector} of $\overrightarrow{G}$ be the vectors
\[
\wt{G} = 
\begin{pmatrix}
    wt(v_{1})\\
    wt(v_{2})\\
    \vdots\\
    wt(v_{n})
\end{pmatrix}
\quad \text{and} \quad
\imb{G} = 
\begin{pmatrix}
    imb(v_{1})\\
    imb(v_{2})\\
    \vdots\\
    imb(v_{n})
\end{pmatrix},
\]  
respectively. 
Figure \ref{fig:ex} illustrates the weight vector, imbalance vector and skew-adjacency matrix for $\overrightarrow{W_4}$ with an arbitrary labeling $f$.

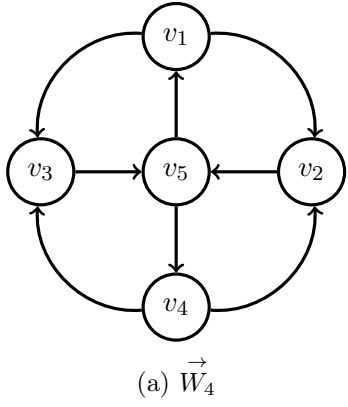
\begin{figure}[H]
    \centering
    \begin{subfigure}{0.45\textwidth}
        \centering
        \begin{tikzpicture}
        [node distance=1.8cm,base node/.style={circle,draw,minimum size=25pt}, scale = 0.4]
            \node[base node, very thick] (5){$v_{5}$};
            \node[base node, very thick][left of=5] (3){$v_{3}$};
            \node[base node, very thick][above of=5] (1){$v_{1}$};
            \node[base node, very thick][below of=5] (4){$v_{4}$};
            \node[base node, very thick][right of=5] (2){$v_{2}$};
            \path [->, very thick] (3) edge node[left] {} (5);
            \path [<-, very thick] (5) edge node[right] {} (2);
            \path [->, very thick](1) edge[bend right=50] node[left] {} (3);
            \path [->, very thick](4) edge[bend left=50] node[left] {} (3);
            \path [->, very thick](4) edge[bend right=50] node[left] {} (2);
            \path [->, very thick](1) edge[bend left=50] node[left] {} (2);
            \path [->, very thick] (5) edge node[left] {} (1);
            \path [->, very thick] (5) edge node[left] {} (4);
        \end{tikzpicture}
        \caption{$\ar{W_4}$} \label{fig:labels}
    \end{subfigure}
    \hspace{0.5cm}
    \begin{subfigure}{0.45\textwidth}
        \begin{align*}
        \wt{W_4} = %
        \begin{pmatrix}
            f(v_{5})-f(v_{2})-f(v_{3})\\
            f(v_{1})+f(v_{4})-f(v_{5})\\
            f(v_{1})+f(v_{4})-f(v_{5})\\
            f(v_{5})-f(v_{2})-f(v_{3})\\
            f(v_{2})+f(v_{3})-f(v_{1})-f(v_{4})
        \end{pmatrix}
        \end{align*}
        \caption{Weight vector for $\ar{W_4}$} \label{fig:weights}
    \end{subfigure}
    
\vspace{1em}
    
    \begin{subfigure}{0.45\textwidth}
        \begin{align*}
        \imb{W_4} = %
        \begin{pmatrix}
            -1\\
            1\\
            1\\
            -1\\
            0
        \end{pmatrix}
        \end{align*}
        \caption{Imbalance vector for $\ar{W_4}$} \label{fig:imb}
    \end{subfigure}
    \hspace{0.5cm}
    \begin{subfigure}{0.45\textwidth}
        \begin{align*}
        S_{\overrightarrow{W_4}} = %
        \begin{pmatrix}
        0 & -1 & -1 & 0 & 1 \\
        1 & 0  & 0  & 1 & -1 \\
        1 & 0  & 0  & 1 & -1 \\
        0 & -1 & -1 & 0 & 1 \\
       -1 & 1  & 1  &-1 & 0
        \end{pmatrix}
        \end{align*}
        \caption{Skew-adjacency matrix for $\ar{W_4}$} \label{fig:matrix}
    \end{subfigure}
    \caption{The weight and imbalance vectors of $\ar{W_4}$ along with its skew-adjacency matrix} \label{fig:ex} 
\end{figure}

\begin{lemma}\label{lemma:weightv}
Let $\AR{G}$ be an oriented graph. For any label vector $\lv{x}$ we have that $S\, \lv{x} = \wt{G}$.
 
\end{lemma}
\begin{proof}
Let $A$ and $S$ be the adjacency and skew-adjacency matrices associated to $\AR{G}$, respectively. Let $\lv{x} = [f(v_1),\dots,f(v_n)]^{T}$ be the label vector. Then

\begin{align*}
    S \lv{x} =%
    S
    \begin{pmatrix}
        f(v_{1}) \\
        \vdots \\
       f(v_{n}) 
    \end{pmatrix}
    = 
    A^{T}
    \begin{pmatrix}
        f(v_{1}) \\
        \vdots \\
       f(v_{n}) 
    \end{pmatrix}
    -
    A
    \begin{pmatrix}
        f(v_{1}) \\
        \vdots \\
       f(v_{n})
    \end{pmatrix}.    
\end{align*}

For a given vertex $v_i$, let $r(i)=|N^{+}(v_i)|$ and $s(i)=|N^{-}(v_i)|$. Enumerate the in- and out-neighborhoods of $v_i$ as $N^{+}(v_{i}) = \{v_{i_{1}}, \dots, v_{i_{r(i)}}\}$ and $N^{-}(v_{i}) = \{v_{i'_{1}}, \dots, v_{i'_{s(i)}}\}$.
Observe that the $i$-th row of $A^{T}$ has ones in the entries $a_{i,i_{1}}, \dots, a_{i,i_{r(i)}}$. Therefore, the $i$-th entry of $A^{T}\lv{x}$ equals $\displaystyle \sum_{j=1}^{r(i)} f(v_{i_{j}})$.
Similarly, the $i$-th row of $A$ has ones in the entries $a_{i,i'_{1}}, \dots, a_{i,i'_{s(i)}}$. Therefore, the $i$-th entry of $A\lv{x}$ equals $\displaystyle \sum_{j=1}^{s(i)} f(v_{i'_{j}})$.

Hence,
\begin{align*}
    S \lv{x}
    = 
    \begin{pmatrix}
        \displaystyle \sum_{j=1}^{r(1)} f(v_{1_{j}}) - \sum_{j=1}^{s(1)} f(v_{1'_{j}}) \\
        \vdots \\
       \displaystyle \sum_{j=1}^{r(n)} f(v_{n_{j}}) - \sum_{j=1}^{s(n)} f(v_{n'_{j}}) 
    \end{pmatrix} 
    =
    \begin{pmatrix}
        wt(v_{1})\\
        \vdots \\
        wt(v_{n}) 
    \end{pmatrix}.
\end{align*}
\end{proof}

\begin{corollary}\label{coro:imbalancev}
Let $\AR{G}$ be an oriented graph, then $S\mathbf{1} = \imb{G}$.

\end{corollary}
\begin{proof}
Consider the labeling $f:V(\AR{G}) \to \mathbb Z^{+}$ given by $f(v_{i}) = 1$ for all $i=1,\dots,n$. From Lemma \ref{lemma:weightv} we have that 
\begin{align*}
    S\mathbf{1}
    =
    S
    \begin{pmatrix}
        f(v_{1}) \\
        \vdots \\
        f(v_{n}) 
    \end{pmatrix}
    = 
    \begin{pmatrix}
        r(1) - s(1) \\
        \vdots \\
        r(n) - s(n) 
    \end{pmatrix} 
    =
    \begin{pmatrix}
        imb(v_{1})\\
        \vdots \\
        imb(v_{n}) 
    \end{pmatrix}.
\end{align*}
\end{proof}

Using this notation, for any oriented graph $\overrightarrow{G}$ we see, 
\begin{equation*}
        \begin{split}
            \sum_{v \in V(\ar{G})}imb(v)&=\mathbf{1}^{T} \imb{G} 
            \\
            &= \mathbf{1}^{T} S \mathbf{1} \\
            &= \mathbf{1}^{T} (-S^{T}) \mathbf{1} \\
            &= -(S\mathbf{1})^{T}\mathbf{1} \\
            &=-\imb{G}^{T}\mathbf{1} \\
            &= -\sum_{v \in V(\ar{G})}imb(v).
        \end{split}
    \end{equation*}

Thus, we get the following:

\begin{observation}\label{obs:balancesum}
    In any oriented graph $\ar{G}$, $\displaystyle \sum_{v \in V(\ar{G})}imb(v)=0$. 
\end{observation}

\begin{lemma}\label{lem:balancelabelsum}
    Let $\AR{G}$ be a DDMOG with $V(\AR{G})=\{v_1,\dots, v_n\}$ and DDM labeling $f:V(\AR{G})\rightarrow \{1, 2, \dots, n\}$. Then
    \begin{align}
        \sum_{i=1}^{n} imb(v_{i})f(v_{i}) = 0 
    \end{align}
\end{lemma}
\begin{proof}
Observe that by Corollary \ref{coro:imbalancev}
\begin{align*}
    \sum_{i=1}^{n}imb(v_{i}) f(v_{i}) = \lv{x}^{T} \imb{G} = \lv{x}^{T} S \mathbf{1}
\end{align*}
where $\lv{x}$ is the label vector induced by $f$.

Since $\overrightarrow{G}$ is DDMOG, we have that $(A^{T}-A)\lv{x} = \mathbf{0}$, see \cite{FM}. Then
\begin{align*}
    \mathbf{0}^{T} = \lv{x}^{T}(A^{T}-A)^{T} = \lv{x}^{T}(A-A^{T}) = \lv{x}^{T} (-S)
\end{align*}
Thus $\lv{x}^{T} S = \mathbf{0}^{T}$, which proves the statement above.   
\end{proof}

With these results, we introduce our first way to build a new DDMOG from an existing DDMOG. In particular, for any DDMOG with imbalance 1 we can create a bigger and now balanced DDMOG.   This technique is illustrated in Figure \ref{fig:ReducingImb}.

\begin{theorem}
    If $\ar{G}$ is a DDMOG with $imb(\ar{G})=1$ and order $n$, then there exists a DDMOG $\ar{G'}$ with $imb(\ar{G'})=0$ and order $n+1$ containing $\ar{G}$ as an induced subgraph.\label{thm:imb1}
\end{theorem}

\begin{proof}
    Since $\ar{G}$ is an oriented graph of imbalance 1, there exists a partition of $V(\ar{G})=\{S_0,S_+,S_-\}$ such that $S_0=\{v \in V(\ar{G}): imb(v)=0\}$, $S_+=\{v \in V(\ar{G}): imb(v)=1\}$, and $S_-=\{v \in V(\ar{G}): imb(v)=-1\}$.  Also note that by Observation \ref{obs:balancesum}, $|S_+|=|S_-|$.  Let $\ar{G'}$ be the graph with vertex set $V(\ar{G'})=V(\ar{G}) \cup \{v'\}$ and edge set $E(\ar{G'})=E(\ar{G}) \cup \{v'v:v \in S_-\} \cup \{vv':v \in S_+\}$. Note that $\ar{G} \leq \ar{G'}$, and moreover, that $\ar{G}$ is an induced subgraph of $\ar{G'}$.  Furthermore, due to the construction of $N_{\ar{G'}}^+(v')$ and  $N_{\ar{G'}}^-(v')$, it follows that $imb(\ar{G'})=0$.  Now let $f:V(\ar{G})\rightarrow \{1,2,\dots,n\}$ be the DDM labeling of $\ar{G}$ and define $f':V(\ar{G'})\rightarrow \{1,2,\dots,n+1\}$ such that $f'(v')=1$ and for $v \in V(\ar{G})$, $f'(v)=f(v)+1$.  
    
    First consider $v \in S_0$.  Since $\ar{G}$ is DDMO and $imb(v)=0$,
    \begin{equation*}
        \begin{split}
            wt_{f'}(v) &=\sum_{u \in N_{\ar{G'}}^+(v)} f'(u)-\sum_{u \in N_{\ar{G'}}^-(v)} f'(u) \\
            &=\sum_{u \in N_{\ar{G}}^+(v)} (f(u)+1)-\sum_{u \in N_{\ar{G}}^-(v)} (f(u)+1) \\
            &=\left(\sum_{u \in N_{\ar{G}}^+(v)} f(u)-\sum_{u \in N_{\ar{G}}^-(v)} f(u)\right)+imb(v)=0.
        \end{split}
    \end{equation*}

    Next consider $v \in S_+$, and note that $N_{\ar{G'}}^+(v)=N_{\ar{G}}^+(v)$ and $N_{\ar{G'}}^-(v)=N_{\ar{G}}^-(v) \cup \{v'\}$.  Since $\ar{G}$ is DDMO and $imb(v)=1$,
    \begin{equation*}
        \begin{split}
            wt_{f'}(v) &=\sum_{u \in N_{\ar{G'}}^+(v)} f'(u)-\sum_{u \in N_{\ar{G'}}^-(v)} f'(u) \\
            &=\sum_{u \in N_{\ar{G}}^+(v)} (f(u)+1)-\left(f'(v')+\sum_{u \in N_{\ar{G}}^-(v)} (f(u)+1)\right)\\
            &=\left(\sum_{u \in N_{\ar{G}}^+(v)} f(u)-\sum_{u \in N_{\ar{G}}^-(v)} f(u)\right)-1+imb(v)=0.
        \end{split}
    \end{equation*}
    Similarly, for $v \in S_-$, since $N_{\ar{G'}}^-(v)=N_{\ar{G}}^-(v)$ and $N_{\ar{G'}}^+(v)=N_{\ar{G}}^+(v) \cup \{v'\}$, it follows that $wt_{f'}(v)=0$.

    Finally, consider $v'$ and note that $N_{\ar{G'}}^+(v')=S_+$ and $N_{\ar{G'}}^-(v')=S_-$.  By Lemma \ref{lem:balancelabelsum},
    \[\sum_{u \in S_+}imb(u)f(u)+\sum_{u \in S_-}imb(u)f(u)=\sum_{u \in V\left(\ar{G}\right)}imb(u)f(u)=0.\]
    Since $\forall u\in S_+$ $imb(u)=1$ and $\forall u\in S_-$ $imb(u)=-1$ it follows that
    \begin{equation*}
        \begin{split}
            wt_{f'}(v') &=\sum_{u \in N_{\ar{G'}}^+(v')}f'(u)-\sum_{u \in N_{\ar{G'}}^-(v')}f'(u) \\
            &=\sum_{u \in S_+}(f(u)+1)-\sum_{u \in S_-}(f(u)+1) \\
            &=\sum_{u \in S_+}imb(u)f(u)+\sum_{u \in S_-}imb(u)f(u)+|S_+|-|S_-| \\
            &=\sum_{u \in V\left(\ar{G}\right)}imb(u)f(u)+|S_+|-|S_-|=0.
        \end{split}
    \end{equation*}
    Thus for each $v \in V(\ar{G'})$, $wt_{f'}(v)=0$, and so $\ar{G'}$ is a DDMOG of order $n+1$ and imbalance 0 where $\ar{G}$ is an induced subgraph of $\ar{G'}$.
\end{proof}

\section{Existence of arbitrarily large connected DDMOGs}\label{arbitrarily_large}

In this section, we create DDMOGs of any order from previously known DDMOGs. We rely on a technique called vertex coalescence introduced in \cite{coal}. We often also need a balanced DDMOG as part of these constructions. Although we will use $C_6(1,2)$ for all our examples of a balanced DDMOG, there are other examples including an orientation of $K_{4s,4s}$ for any $s \geq 1$ as given in \cite{FM}.

\begin{figure}[h]
    \centering
		  \begin{subfigure}[t]{.5\textwidth}
  \centering
  \begin{tikzpicture}[scale=0.8,node distance=1.4cm, base node/.style={circle,draw,minimum size=25pt}]

\node[base node, very thick] (5) at (1, 4) {5};
\node[text=red] at (1.9,3.8) {0};

\node[base node, very thick] (1) at (-2, 5.5) {1};
\node[text=red] at (-2,6.4) {-1};
\node[base node, very thick] (2) at (0, 5.5) {2};
\node[text=red] at (-0.6,6) {1};
\node[base node, very thick] (3) at (2,5.5) {3};
\node[text=red] at (2.6,6) {1};
\node[base node, very thick] (4) at (4,5.5) {4};
\node[text=red] at (4,6.4) {-1};

\path[->, very thick] (5) edge (1);
\path[->, very thick] (2) edge (5);
\path[->, very thick] (3) edge  (5); %
\path[->, very thick] (5) edge  (4); %

\path[->, very thick] (1) edge (2);
\path[->, very thick] (1) edge [bend left=39]  (3);
\path[->, very thick] (4) edge[bend right=39]  (2);
\path[->, very thick] (4) edge (3);

\end{tikzpicture}
  \caption{$imb(\ar{G})=1$}
  \end{subfigure}
		  \begin{subfigure}[t]{.49\textwidth}
  \centering
  \begin{tikzpicture}[scale=0.8,node distance=1.4cm, base node/.style={circle,draw,minimum size=25pt}]

\node[base node, very thick] (6) at (1, 7.5) {1}; 
\node at (1.7,8.1) {$v'$};

\node[base node, very thick] (5) at (1, 4) {6}; 

\node[base node, very thick] (1) at (-2, 5.5) {2}; 
\node[base node, very thick] (2) at (0, 5.5) {3}; 
\node[base node, very thick] (3) at (2,5.5) {4}; 
\node[base node, very thick] (4) at (4,5.5) {5}; 

\path[->, very thick] (5) edge (1);
\path[->, very thick] (2) edge (5);
\path[->, very thick] (3) edge  (5); %
\path[->, very thick] (5) edge  (4); %

\path[->, very thick] (1) edge (2);
\path[->, very thick] (1) edge [bend left=39]  (3);
\path[->, very thick] (4) edge[bend right=39]  (2);
\path[->, very thick] (4) edge (3);

\path[->, very thick] (6) edge[bend right=30]  (1);
\path[<-, very thick] (6) edge (2);
\path[<-, very thick] (6) edge (3);
\path[->, very thick] (6) edge[bend left=30] (4);

\end{tikzpicture}
  \caption{$imb(\ar{G'})=0$}
  \end{subfigure}\caption{The process of reducing imbalance in Theorem \ref{thm:imb1}}\label{fig:ReducingImb}
			\end{figure}
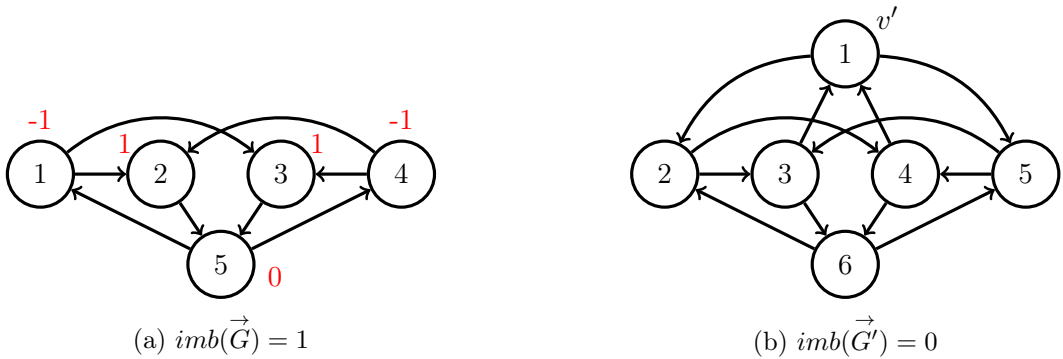

\begin{definition}  

Consider a vertex $u$ from an oriented graph $\overrightarrow{G_1}$ and a vertex $v$ from an oriented graph $\overrightarrow{G_2}$. The \textit{vertex coalescence} of $\overrightarrow{G_1}$ and $\overrightarrow{G_2}$ via $u$ and $v$ denoted $\overrightarrow{G_1} \cdot_{uv} \overrightarrow{G_2}$ results from adding all edges $(u,x)$ to $\overrightarrow{G_1} \cup \overrightarrow{G_2}$ where $(v,x)$ is an edge in $\overrightarrow{G_2}$, all edges $(x,u)$ to $\overrightarrow{G_1} \cup \overrightarrow{G_2}$ where $(x,v)$ is an edge in $\overrightarrow{G_2}$, and then deleting $v$ from the resulting graph. 
\end{definition}

Besides vertex coalescence, we also rely on the assistance of Lemma~\ref{lem:shift} to grow larger DDMOGs.  This lemma shows that by shifting each label in a balanced oriented graph by the same constant $s$, the weight of each vertex remains unchanged in the resulting labeling.

\begin{lemma}\label{lem:shift}
Let $s \in \mathbb Z$, $\ar{G}$ be a balanced oriented graph of order $n$ with labeling $f:V(\ar{G})\rightarrow \mathbb Z^+$, and $h$ be the labeling $h: V(\ar{G})\rightarrow \mathbb Z^+$ given by $h(v)=f(v)+s$ for all $v \in V(\ar{G})$.  Then for all $v \in V(\ar{G})$, $wt_f(v)=wt_h(v)$. 
\end{lemma}

\begin{proof}
    Let $v \in V(\ar{G})$ be arbitrary.  Since $\ar{G}$ is balanced we know $imb(\ar{G})=0$, and
    \begin{equation*}
        \begin{split}
            wt_h(v)&=\sum_{u \in N^+_{\ar{G}}(v)}h(u)-\sum_{u \in N^-_{\ar{G}}(v)}h(u)\\
            &=\sum_{u \in N^+_{\ar{G}}(v)}(f(u)+s)-\sum_{u \in N^-_{\ar{G}}(v)}(f(u)+s)\\
            &=\left(\sum_{u \in N^+_{\ar{G}}(v)}f(u)-\sum_{u \in N^-_{\ar{G}}(v)}f(u)\right)+s\left(|N_{\ar{G}}^+(v)|-|N_{\ar{G}}^-(v)|\right)\\
            &=wt_f(v)+s\cdot imb(\ar{G})\\
            &=wt_f(v).
        \end{split}
    \end{equation*}
\end{proof}

In Theorem \ref{thm:chain}, we create a new DDMOG by combining an initial DDMOG with a balanced DDMOG using vertex coalescence. 

\begin{theorem} \label{thm:chain}
Let $\ar{G}$ be a DDMOG of order $n$, let $g$ be a DDM labeling of $\ar{G}$, and let $u \in V(\ar{G})$ be such that $g(u)=n$.  Let $\ar{H}$ be a DDMOG of order $m$, let $h$ be a DDM labeling of $\ar{H}$, and let $w \in V(\ar{H})$ be such that $h(w)=1$.  If $\ar{H}$ is balanced, then $\overrightarrow{K}=\overrightarrow{G} \cdot_{uw} \overrightarrow{H}$ is a DDMOG on $n+m-1$ vertices.   
\end{theorem}

\begin{proof} First, define a labeling $h_{n-1}$ of $\AR{H}$ such that $h_{n-1}(v)=h(v)+n-1$ for each $v \in V(\AR{H})$.  Next note that $V(\ar{K})=(V(\ar{G})\cup V(\ar{H}))\setminus\{w\}$, and define a labeling $f: V(\ar{K})\rightarrow \{1,2,\dots,m+n-1\}$ by
\[f(v)=\begin{cases}g(v) &  \text{ if }v \in V(\ar{G})\\
h_{n-1}(v) & \text{ if }v \in V(\ar{H})\setminus \{w\}.\end{cases}\]
Since $N_{\ar{K}}^+(v)=N_{\ar{G}}^+(v)$ and $N_{\ar{K}}^-(v)=N_{\ar{G}}^-(v)$ for all $v \in V(\ar{G})\setminus \{u\}$ and $f(v)=g(v)$ for all $v \in V(\ar{G})$, it follows that for $v \in V(\ar{G})\setminus \{u\}$, $wt_f(v)=wt_g(v)=0$ with the second equality coming from the fact that $g$ is a DDM labeling of $\ar{G}$. 

Next note that for all vertices $v \in V(\ar{H})\setminus \{w\}$, $w \in N_{\ar{H}}^+(v)$ if and only if $N_{\ar{K}}^+(v)=(N_{\ar{H}}^+(v) \cup \{u\})\setminus \{w\}$, and otherwise when $w \not\in N_{\ar{H}}^+(v)$, $N_{\ar{K}}^+(v)=N_{\ar{H}}^+(v)$. Similarly  $w \in N_{\ar{H}}^-(v)$ if and only if $N_{\ar{K}}^-(v)=(N_{\ar{H}}^-(v) \cup \{u\})\setminus \{w\}$, and otherwise when $w \not\in N_{\ar{H}}^-(v)$, $N_{\ar{K}}^-(v)=N_{\ar{H}}^-(v)$.  
Since $\ar{H}$ is balanced it follows from Lemma \ref{lem:shift} that $wt_{h_{n-1}}(v)=wt_h(v)=0$ for all $v \in V(\AR{H})$, with the second equality coming from the fact that $h$ is a DDM labeling of $\ar{H}$.  Furthermore, since $f(v)=h_{n-1}(v)$ for all $v \in V(\AR{H}) \setminus \{w\}$, and since $f(u)=g(u)=n=1+(n-1)=h(w)+n-1=h_{n-1}(w)$, it follows that for $v \in V(\AR{H}) \setminus \{w\}$, $wt_f(v)=wt_{h_{n-1}}(v)=0$.

Finally, since
 $N_{\ar{K}}^+(u)=N_{\ar{G}}^+(u) \cup N_{\ar{H}}^+(w)$ and $N_{\ar{K}}^-(u)=N_{\ar{G}}^-(u)\cup N_{\ar{G}}^+(w)$ we have $wt_f(u)=wt_g(u)+wt_{h_{n-1}}(w)=wt_g(u)+wt_h(w)$. Given that $g$ and $h$ are DDM labelings of $\ar{G}$ and $\ar{H}$ respectively we get $wt_f(u)=wt_g(u)+wt_h(w)=0$.  Since $f:V(\ar{K})\rightarrow \{1,2,\dots,n+m-1\}$ is a bijection and $wt_f(v)=0$ for all $v \in V(\ar{K})$, $\ar{K}=\overrightarrow{G} \cdot_{uw} \overrightarrow{H}$ is a DDMOG on $n+m-1$ vertices.
\end{proof}

Figure~\ref{fig:exampleLace} gives us an example of the type of construction given in Theorem \ref{thm:chain} where the vertex labeled 5 in $\ar{W_4}$ is coalesced to the vertex labeled 1 in $\ar{C_6}(1,2)$.

\begin{figure}[ht]
\centering

\begin{minipage}[b]{0.45\textwidth}  \centering
  \begin{subfigure}[t]{\textwidth}
  \centering
  \begin{tikzpicture}[scale=0.8,node distance=1.4cm, base node/.style={circle,draw,minimum size=25pt}]

\node[base node, very thick] (5) at (1, 4) {5};
\node at (1.9,3.8) {$u$};


\node[base node, very thick] (1) at (-2, 5.5) {1};
\node[base node, very thick] (2) at (0, 5.5) {2};
\node[base node, very thick] (3) at (2,5.5) {3};
\node[base node, very thick] (4) at (4,5.5) {4};

\path[->, very thick] (5) edge (1);
\path[->, very thick] (2) edge (5);
\path[->, very thick] (3) edge  (5); %
\path[->, very thick] (5) edge  (4); %

\path[->, very thick] (1) edge (2);
\path[->, very thick] (1) edge [bend left=39]  (3);
\path[->, very thick] (4) edge[bend right=39]  (2);
\path[->, very thick] (4) edge (3);



\end{tikzpicture}
  \caption{A DDM Labeling of $\overrightarrow{W_4}$}
  \label{fig:a}
  \end{subfigure}
	
  \vspace{0.5cm} 
	
  \begin{subfigure}[b][.21\textheight][t]{\textwidth}  \centering
  \begin{tikzpicture}[scale=0.8,node distance=1.4cm, base node/.style={circle,draw,minimum size=25pt}]

\node[base node, very thick] (5) at (1, 4) {1}; 
\node at (1.9,4.5) {$w$};

\node[base node, very thick] (6) at (-2, 2.5) {2}; 
\node[base node, very thick] (7) at (0, 2.5) {3}; 
\node[base node, very thick] (8) at (2, 2.5) {4};  
\node[base node, very thick] (9) at (4, 2.5) {5}; 


\node[base node, very thick] (10) at (1, 1) {6}; 

\path[->, very thick] (6) edge [bend left=15] (5);
\path[->, very thick] (5) edge [bend right=10] (7);
\path[->, very thick] (5) edge [bend left=10] (8);

\path[->, very thick] (7) edge (6);
\path[->, very thick] (8) edge[bend right=30] (6); 
\path[->, very thick] (7) edge[bend left=30] (9); 
\path[->, very thick] (8) edge (9);
\path[->, very thick] (9) edge[bend right=20] (5);  



\path[->, very thick] (10) edge (7);
\path[->, very thick] (10) edge (8);
\path[->, very thick] (9) edge (10);
\path[->, very thick] (6) edge (10);


\end{tikzpicture}
  \caption{A labeling of $\overrightarrow{C_6}(1,2)$}
  \label{fig:b}
  \end{subfigure}
\end{minipage}
\hfill
\begin{minipage}[b]{0.45\textwidth}
\begin{center}
  \begin{subfigure}[t]{\textwidth}
  \centering
  \begin{tikzpicture}[scale=0.9,node distance=1.4cm, base node/.style={circle,draw,minimum size=25pt}]

\node[base node, very thick] (5) at (1, 4) {5};
\node[base node, very thick] (6) at (-2, 2.5) {6};
\node[base node, very thick] (7) at (0, 2.5) {7};
\node[base node, very thick] (8) at (2, 2.5) {8};
\node[base node, very thick] (9) at (4, 2.5) {9};

\node[base node, very thick] (1) at (-2, 5.5) {1};
\node[base node, very thick] (2) at (0, 5.5) {2};
\node[base node, very thick] (3) at (2,5.5) {3};
\node[base node, very thick] (4) at (4,5.5) {4};

\node[base node, very thick] (10) at (1, 1) {10};

\path[->, very thick] (6) edge [bend left=15] (5);
\path[->, very thick] (5) edge [bend right=10] (7);
\path[->, very thick] (5) edge [bend left=10] (8);

\path[->, very thick] (7) edge (6);
\path[->, very thick] (8) edge[bend right=30] (6); 
\path[->, very thick] (7) edge[bend left=30] (9); 
\path[->, very thick] (8) edge (9);
\path[->, very thick] (9) edge[bend right=20] (5);  

\path[->, very thick] (5) edge (1);
\path[->, very thick] (2) edge (5);
\path[->, very thick] (3) edge  (5); %
\path[->, very thick] (5) edge  (4); %

\path[->, very thick] (1) edge (2);
\path[->, very thick] (1) edge [bend left=39]  (3);
\path[->, very thick] (4) edge[bend right=39]  (2);
\path[->, very thick] (4) edge (3);

\path[->, very thick] (10) edge (7);
\path[->, very thick] (10) edge (8);
\path[->, very thick] (9) edge (10);
\path[->, very thick] (6) edge (10);


\end{tikzpicture}
  \caption{	DDM Labeling of $\overrightarrow{W_4} \cdot_{uw} \overrightarrow{C_6}(1,2)$}
  \label{fig:c}
  \end{subfigure}\end{center}
		\vspace{6mm}
\end{minipage}

\caption{Illustrating the DDM labeling used in the proof of Theorem \ref{thm:chain}}\label{fig:exampleLace}
\end{figure}
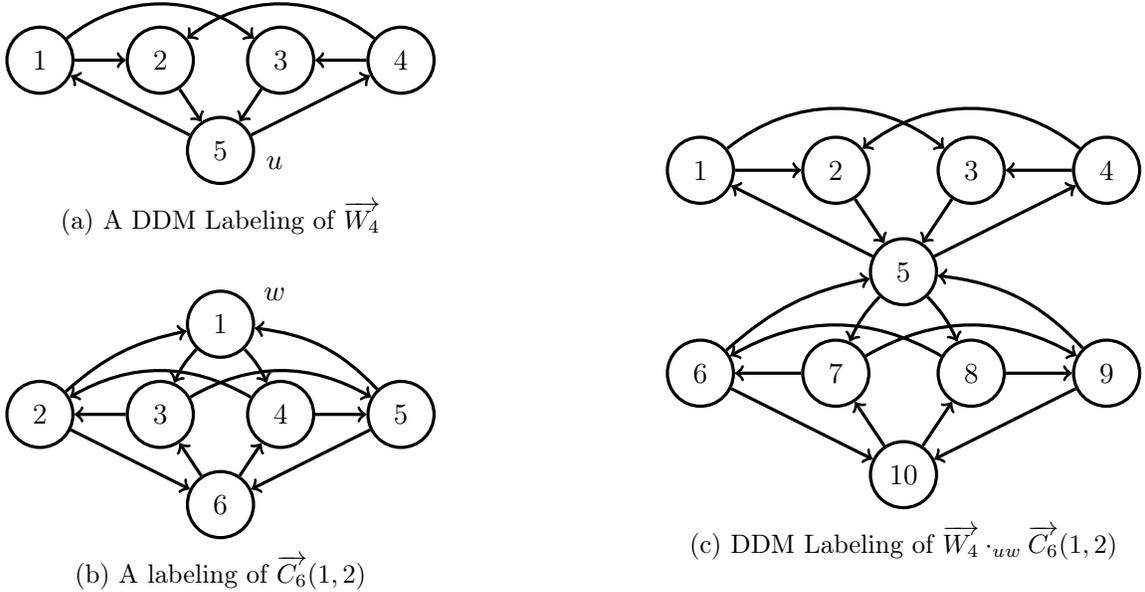

\begin{theorem}\label{thm:DDMOGforAllNAtLeast5}
For all $n\geq 5$, there exists a DDMOG on $n$ vertices.  
\end{theorem}

\noindent {\bf Proof:} Figure \ref{fig:DDMOGs} gives DDMOGs on 5, 6, 7, 8, and 9 vertices, $\AR{R_5}, \AR{R_6}, \AR{R_7}, \AR{R_8}, \AR{R_9}$, respectively.

In this proof, we will always coalesce the vertex with the highest label in the first graph with the vertex labeled 1 in the second graph. For the sake of notation, we will not put the vertices being coalesced in the notation. Per Theorem \ref{thm:chain}, $\AR{R_i} \cdot \AR{R_6}$ produces a DDMOG on $i+5$ vertices. Beginning with $R_i$ and iteratively performing a vertex coalescence with $\AR{R_6}$ a total of $j$ times generates a DDMOG on $5j+i$ vertices. Since $5 \leq i \leq 9$ and $0 \leq j < \infty$, this generates DDMOGs on $n$ vertices for every $n \geq 5$. \qed 

Figure ~\ref{fig:exampleLace} gives us a DDMOG on 10 vertices (namely $\AR{R_5} \cdot \AR{R_6}$).   Figure \ref{fig:gluingToGetAllN} gives an example on 15 vertices by coalescing twice, creating $(\AR{R_5} \cdot \AR{R_6}) \cdot \AR{R_6}$.

\begin{figure}[H]
    \centering
    \begin{subfigure}{0.45\textwidth}
        \centering
        \begin{tikzpicture}
        [node distance=1.8cm,base node/.style={circle,draw,minimum size=25pt}, scale=0.8]
            \node[base node, very thick] (5){5};
            \node[base node, very thick][left of=5] (3){3};
            \node[base node, very thick][above of=5] (1){1};
            \node[base node, very thick][below of=5] (4){4};
            \node[base node, very thick][right of=5] (2){2};
            \path [->, very thick] (3) edge node[left] {} (5);
            \path [<-, very thick] (5) edge node[right] {} (2);
            \path [->, very thick](1) edge[bend right=50] node[left] {} (3);
            \path [->, very thick](4) edge[bend left=50] node[left] {} (3);
            \path [->, very thick](4) edge[bend right=50] node[left] {} (2);
            \path [->, very thick](1) edge[bend left=50] node[left] {} (2);
            \path [->, very thick] (5) edge node[left] {} (1);
            \path [->, very thick] (5) edge node[left] {} (4);
        \end{tikzpicture}
        \caption{$\AR{R_{5}}=\AR{W_4}$} \label{$W_4$ figure}
    \end{subfigure}
    \hspace{0.5cm}
    \begin{subfigure}{0.45\textwidth}
        \centering
        \begin{tikzpicture}
        [node distance=1.4cm,base node/.style={circle,draw,minimum size=25pt}, scale=0.9]
            \node[base node, very thick] (v1) at (90:2) {2};
            \node[base node, very thick] (v2) at (30:2) {3};
            \node[base node, very thick] (v3) at (-30:2) {6};
            \node[base node, very thick] (v4) at (-90:2) {5};
            \node[base node, very thick] (v5) at (-150:2) {4};
            \node[base node, very thick] (v6) at (150:2) {1};
            
            \path[->, very thick] (v1) edge  (v6); 
            \path[->, very thick] (v1) edge (v3);  
            
            \path[->, very thick] (v2) edge (v1);
            \path[->, very thick] (v2) edge (v4);
            
            \path[->, very thick] (v3) edge (v2);
            \path[->, very thick] (v3) edge (v5);
            
            \path[->, very thick] (v4) edge (v3);
            \path[->, very thick] (v4) edge (v6);
            
            \path[->, very thick] (v5) edge (v4);
            \path[->, very thick] (v5) edge (v1);
            
            \path[->, very thick] (v6) edge (v5);
            \path[->, very thick] (v6) edge (v2);
        \end{tikzpicture}
        \caption{$\AR{R_{6}}=\AR{C_6}(1,2)$}
        \label{fig:circulantC612}
    \end{subfigure}
    
\vspace{1em}
    
    \begin{subfigure}{0.45\textwidth}
        \centering
        \begin{tikzpicture}[node distance=0.5cm,base node/.style={circle,draw,minimum size=25pt}, ->, >=stealth, very thick, scale=0.8]
            \node[base node] (v1) at (90:3) {1};
            \node[base node] (v2) at ({90-360/7}:3) {2};
            \node[base node] (v3) at ({90-2*360/7}:3) {3};
            \node[base node] (v4) at ({90-3*360/7}:3) {4};
            \node[base node] (v5) at ({90-4*360/7}:3) {5};
            \node[base node] (v6) at ({90-5*360/7}:3) {6};
            \node[base node] (v7) at ({90-6*360/7}:3) {7};
            \path (v1) edge (v4);
            \path (v1) edge (v7);
            \path (v2) edge (v1);
            \path (v2) edge (v6);
            \path (v3) edge (v1);
            \path (v3) edge (v2);
            \path (v3) edge (v6);
            \path (v4) edge (v2);
            \path (v4) edge (v3);
            \path (v4) edge (v7);
            \path (v5) edge (v3);
            \path (v5) edge (v4);
            \path (v6) edge (v1);
            \path (v6) edge (v4);
            \path (v7) edge (v5);      
        \end{tikzpicture}        
        \caption{$\AR{R_{7}}$} \label{fig:G7}
    \end{subfigure}
    \hspace{0.5cm}
    \begin{subfigure}{0.45\textwidth}
        \centering
        \begin{tikzpicture}[node distance=0.5cm,base node/.style={circle,draw,minimum size=25pt}, ->, >=stealth, very thick, scale=0.8]
            \node[base node] (v1) at (90:3) {1};
            \node[base node] (v2) at ({90-360/8}:3) {2};
            \node[base node] (v3) at ({90-2*360/8}:3) {3};
            \node[base node] (v4) at ({90-3*360/8}:3) {4};
            \node[base node] (v5) at ({90-4*360/8}:3) {5};
            \node[base node] (v6) at ({90-5*360/8}:3) {6};
            \node[base node] (v7) at ({90-6*360/8}:3) {7};
            \node[base node] (v8) at ({90-7*360/8}:3) {8};
            \path (v1) edge (v5);
            \path (v1) edge (v6);
            \path (v2) edge (v4);
            \path (v2) edge (v5);
            \path (v2) edge (v6);
            \path (v3) edge (v1);
            \path (v3) edge (v4);
            \path (v3) edge (v8);
            \path (v4) edge (v5);
            \path (v5) edge (v7);
            \path (v6) edge (v3);
            \path (v7) edge (v2); 
            \path (v7) edge (v3); 
            \path (v8) edge (v1);
            \path (v8) edge (v2); 
        \end{tikzpicture}        
        \caption{$\AR{R_{8}}$} \label{fig:G8}
    \end{subfigure}

\vspace{1em}

    \begin{subfigure}{0.9\textwidth}
        \centering
        \begin{tikzpicture}[node distance=0.5cm,base node/.style={circle,draw,minimum size=25pt}, ->, >=stealth, very thick, scale=0.8]
            \node[base node] (v1) at (90:3) {1};
            \node[base node] (v2) at ({90-360/9}:3) {2};
            \node[base node] (v3) at ({90-2*360/9}:3) {3};
            \node[base node] (v4) at ({90-3*360/9}:3) {4};
            \node[base node] (v5) at ({90-4*360/9}:3) {5};
            \node[base node] (v6) at ({90-5*360/9}:3) {6};
            \node[base node] (v7) at ({90-6*360/9}:3) {7};
            \node[base node] (v8) at ({90-7*360/9}:3) {8};
            \node[base node] (v9) at ({90-8*360/9}:3) {9};
            \path (v1) edge (v2);
            \path (v1) edge (v8);
            \path (v2) edge (v8);
            \path (v3) edge (v1);
            \path (v3) edge (v5);
            \path (v3) edge (v6);            
            \path (v3) edge (v8);
            \path (v4) edge (v3);
            \path (v4) edge (v7);
            \path (v4) edge (v9);
            \path (v5) edge (v4);
            \path (v5) edge (v7);
            \path (v6) edge (v4);
            \path (v6) edge (v7);
            \path (v7) edge (v1);
            \path (v7) edge (v2); 
            \path (v7) edge (v3);
            \path (v7) edge (v9);
            \path (v8) edge (v4);
            \path (v8) edge (v5);
            \path (v8) edge (v6);
            \path (v9) edge (v3);
            \path (v9) edge (v8);
        \end{tikzpicture}        
        \caption{$\AR{R_{9}}$} \label{fig:G9}
    \end{subfigure}
    \caption{DDMOGs} \label{fig:DDMOGs}
\end{figure}









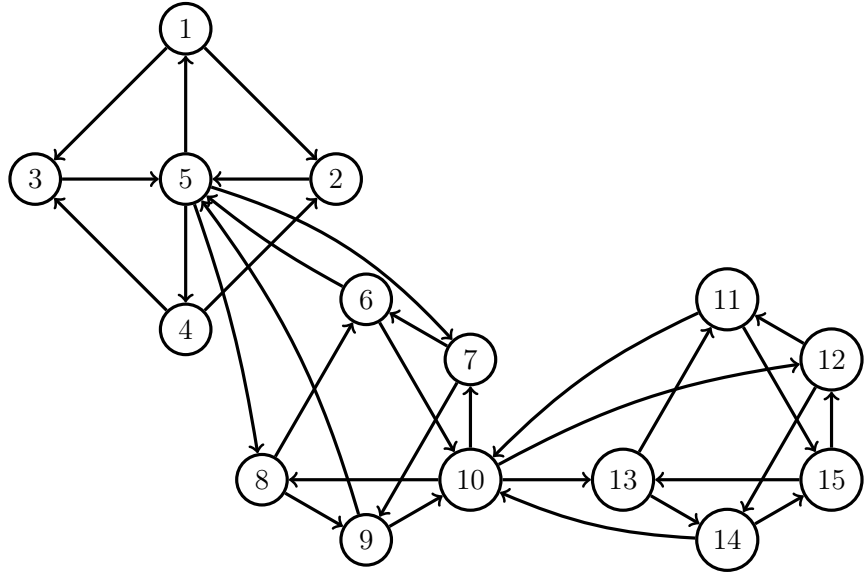
\begin{figure}[H]
\centering
\begin{tikzpicture}
[node distance=2cm, base node/.style={circle,draw,minimum size=15pt},scale=0.8]
\node[base node, very thick] (5){5};
\node[base node, very thick][left of=5] (3){3};
\node[base node, very thick][above of=5] (1){1};
\node[base node, very thick][below of=5] (4){4};
 \node[base node, very thick][right of=5] (2){2};
  \path [->, very thick] (3) edge node[left] {} (5);
   \path [<-, very thick] (5) edge node[right] {} (2);
 \path [->, very thick](1) edge  node[left] {} (3);
 \path [->, very thick](4) edge  node[left] {} (3);
  \path [->, very thick](4) edge  node[left] {} (2);
  \path [->, very thick](1) edge  node[left] {} (2);
\path [->, very thick] (5) edge node[left] {} (1);
\path [->, very thick] (5) edge node[left] {} (4);

\node[base node, very thick] (w1) at ($(90:2) + (3,-4)$) {6};
\node[base node, very thick] (w2) at ($(30:2) + (3,-4)$) {7};
\node[base node, very thick] (w3) at ($(-30:2) + (3,-4)$) {10};
\node[base node, very thick] (w4) at ($(-90:2) + (3,-4)$) {9};
\node[base node, very thick] (w5) at ($(-150:2) + (3,-4)$) {8};

\path[->, very thick] (w1) edge [bend left=5] (5); 
\path[->, very thick] (w1) edge (w3);  
\path[->, very thick] (w2) edge (w1);
\path[->, very thick] (w2) edge (w4);
\path[->, very thick] (w3) edge (w2);
\path[->, very thick] (w3) edge (w5);
\path[->, very thick] (w4) edge (w3);
\path[->, very thick] (w4) edge [bend right=10] (5);
\path[->, very thick] (w5) edge (w4);
\path[->, very thick] (w5) edge (w1);
\path[->, very thick] (5) edge [bend left=5] (w5);
\path[->, very thick] (5) edge [bend left=15]  (w2);

\node[base node, very thick] (z1) at ($(90:2) + (9,-4)$) {11}; 
\node[base node, very thick] (z2) at ($(30:2) + (9,-4)$) {12}; 
\node[base node, very thick] (z3) at ($(-30:2) + (9,-4)$) {15}; 
\node[base node, very thick] (z4) at ($(-90:2) + (9,-4)$) {14}; 
\node[base node, very thick] (z5) at ($(-150:2) + (9,-4)$) {13}; 
 
\path[->, very thick] (z1) edge[bend right=10] (w3); 
\path[->, very thick] (z1) edge (z3);  
\path[->, very thick] (z2) edge (z1);
\path[->, very thick] (z2) edge (z4);
\path[->, very thick] (z3) edge (z2);
\path[->, very thick] (z3) edge (z5);
\path[->, very thick] (z4) edge (z3);
\path[->, very thick] (z4) edge[bend left=10] (w3);
\path[->, very thick] (z5) edge (z4);
\path[->, very thick] (z5) edge (z1);
\path[->, very thick] (w3) edge (z5);
\path[->, very thick] (w3) edge[bend left=10] (z2);
\end{tikzpicture}
 
\caption{A DDM labeling of $(\ar{R_5} \cdot \ar{R_6}) \cdot \ar{R_6}$}
\label{fig:gluingToGetAllN}
\end{figure}

Our technique in Theorem~\ref{thm:DDMOGforAllNAtLeast5} was to grow larger DDMOGs by repeatedly coalescing the highest labeled vertex in an initial DDMOG with the lowest labeled vertex in a balanced oriented graph, namely, $\AR{R_{6}}=\AR{C_6}(1,2)$.   This same technique cannot be used to grow larger DDMOGs using $\AR{R_{5}}=\AR{W_4}$ since $imb(\AR{W_4})=1$.  

On the other hand, by concentrating on repeatedly combining $W_4$'s at a central vertex, we find that another class of graphs formed by vertex coalescence is the windmill of $k$ wheels, denoted as $W_4^{(k)}$. In Theorem \ref{thmWindmillWheel}, we show that $W_4^{(k)}$ is DDMO for $k \geq 1$.

\begin{definition}\label{WWheelK}
     Let $n=1+4k$, $k \in \mathbb{Z}^+$. Let $G$ be the graph on $n$ vertices with
\begin{equation*}
V(G)=\{ v\} \cup \{v_{ij} \,:\, 1\leq i \leq k,\, 1\leq j \leq 4\}, \; \text{and}
\end{equation*}     
\begin{equation*}E(G)=\bigcup_{i=1}^k\Bigl( \{ v_{i1}v_{i2}, v_{i1}v_{i3}, v_{i3}v_{i4}, v_{i2}v_{i4}\}\cup \{vv_{i1}, vv_{i2}, v v_{i3}, vv_{i4}  \}\Bigl).
\end{equation*}
We call this graph $G=W_4^{(k)}$ a windmill of $k$  wheels, with {\it{central}} vertex $v$. An example of this graph (when $k=2$) is given in Figure \ref{fig:handdrawn_digraph1}. \end{definition}

\begin{figure}[H]
\centering
\begin{tikzpicture}
[node distance=1.4cm, base node/.style={circle,draw,minimum size=25pt}]

\node[base node, very thick] (1) at (-3, 2) {$v_{11}$};
\node[base node, very thick] (2) at (-1, 2) {$v_{12}$};
\node[base node, very thick] (7) at (1, 2) {$v_{13}$};
\node[base node, very thick] (8) at (3, 2) {$v_{14}$};

\node[base node, very thick] (3) at (-3, -1) {$v_{21}$};
\node[base node, very thick] (4) at (-1, -1) {$v_{22}$};
\node[base node, very thick] (5) at (1, -1) {$v_{23}$};
\node[base node, very thick] (6) at (3, -1) {$v_{24}$};

\node[base node, very thick] (9) at (0, 0.5) {$v$};

\path[-, very thick] (1) edge (2);
\path[-, very thick] (8) edge[bend right=30] (2);
\path[-, very thick] (8) edge (7);
\path[-, very thick] (1) edge[bend left=30] (7);

\path[-, very thick] (9) edge (1);
\path[-, very thick] (2) edge (9);
\path[-, very thick] (7) edge (9);
\path[-, very thick] (9) edge (8);

\path[-, very thick] (9) edge (3);
\path[-, very thick] (4) edge (9);
\path[-, very thick] (5) edge (9);
\path[-, very thick] (9) edge (6);

\path[-, very thick] (3) edge (4);
\path[-, very thick] (6) edge [bend left=30]  (4);
\path[-, very thick] (6) edge (5);
\path[-, very thick] (3) edge[bend right=30] (5);

\end{tikzpicture}
\caption{Windmill of two wheels, $W_4 ^{(2)}$}
\label{fig:handdrawn_digraph1}
\end{figure}
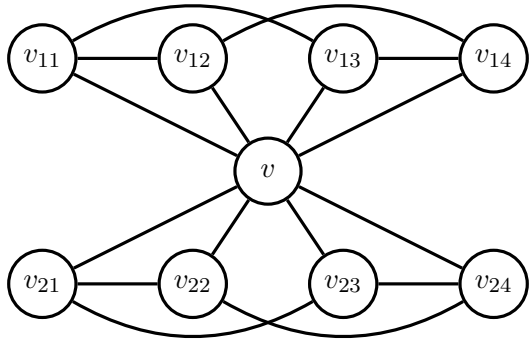

\begin{theorem}\label{thmWindmillWheel}
   $W_4^{(k)}$ is DDMO for all $k\in \mathbb{Z}^+$.

\end{theorem}
 \begin{proof} Let $G = W_4^{(k)}$ for some $k \in \mathbb{Z}^+$, with  central vertex $v$; note that the number of vertices in $G$ is  $n=1+4k$. When $k=1$, we have $G=W_4$, which is known to be DDMO.  Let $k>1$; first, we direct the edges of $\ar{G}$ as follows:
\begin{equation*} E(\ar{G})=\bigcup\limits_{i=1}^k\Bigl( \{v_{i1}v_{i2}, v_{i1}v_{i3}, v_{i4}v_{i3}, v_{i4}v_{i2}\}\cup \{vv_{i1}, v_{i2}v, v v_{i4}, v_{i3}v \}\Bigl).\end{equation*}
We define a bijection $f: V(\AR{G}) \rightarrow \{1,2, \dots, n\}$ by 
\begin{equation}
f(u)=
    \begin{cases}
        n & \text{if } u=v\\
        2i-1 & \text{if } u=v_{i1} \\
        2i & \text{if } u=v_{i2} \\
        n-2i & \text{if } u=v_{i3} \\
        n+1-2i & \text{if } u=v_{i4}, \\
    \end{cases} \; \text{where} \; 1\leq i \leq k.
\end{equation}
Then for all $i \in \{1, \hdots, k\}$ we have 
\begin{center}
{\renewcommand{\arraystretch}{1.5}
\begin{tabular}{ c | c | c  }
Vertex, $u$ & $f(u)$ & $wt_f(u)$ \\
\hline
$v_{i1}$ & $2i-1$ & $n-(2i +(n-2i))=0$ \\ 
$v_{i2}$ & $2i$ & $((2i-1)+(n-2i+1))-n=0$ \\  
 $v_{i3}$ & $n-2i$ & $((2i-1)+(n-2i+1))-n=0$   \\
 $v_{i4}$ & $n-2i+1$ & $n-(2i +(n-2i))=0$ \\  
 $v$ & $n$ & $\sum\limits_{i=1}^{k} (2i +(n-2i)) -((2i-1)+(n-2i+1))=0$
\end{tabular}}
\end{center}
Hence, $f$ is a DDM labeling of the described orientation of $W_4^{(k)}$ for all $k \geq 1$. 

\end{proof} 
An illustration of Theorem~\ref{thmWindmillWheel} (when $k=3$) is given in  Figure~\ref{fig:DDMwheelwindmill_extendedtop}.


\begin{figure}[H]
\centering
\begin{tikzpicture}
[node distance=1.4cm, base node/.style={circle,draw,minimum size=25pt},scale=0.8]

\node[base node, very thick] (1) at (-7, 1.25) {1};
\node[base node, very thick] (2) at (-5, 1.75) {2};
\node[base node, very thick] (11) at (-3, 2) {11};
\node[base node, very thick] (12) at (-1, 2) {12};
\node[base node, very thick] (5) at (1, 2) {5};
\node[base node, very thick] (6) at (3, 2) {6};
\node[base node, very thick] (7) at (5, 1.75) {7};
\node[base node, very thick] (8) at (7, 1.25) {8};

\node[base node, very thick] (3) at (-3, -1) {3};
\node[base node, very thick] (4) at (-1, -1) {4};
\node[base node, very thick] (9) at (1, -1) {9};
\node[base node, very thick] (10) at (3, -1) {10};

\node[base node, very thick] (13) at (0, 0.5) {13};

\path[->, very thick] (1) edge (2);
\path[->, very thick] (12) edge (11);
\path[->, very thick] (8) edge (7);
\path[->, very thick] (5) edge (6);

\path[->, very thick] (12) edge[bend right=30] (2);
\path[->, very thick] (1) edge[bend left=30] (11);
\path[->, very thick] (8) edge[bend right=30] (6);
\path[->, very thick] (5) edge[bend left=30] (7);

\path[->, very thick] (2) edge (13);
\path[->, very thick] (6) edge (13);
\path[->, very thick] (7) edge (13);
\path[->, very thick] (11) edge (13);
\path[->, very thick] (13) edge (12);
\path[->, very thick] (13) edge (8);
\path[->, very thick] (13) edge (1);
\path[->, very thick] (13) edge (5);

\path[->, very thick] (13) edge (3);
\path[->, very thick] (13) edge (10);
\path[->, very thick] (4) edge (13);
\path[->, very thick] (9) edge (13);

\path[->, very thick] (3) edge (4);
\path[->, very thick] (10) edge [bend left=30] (4);
\path[->, very thick] (10) edge (9);
\path[->, very thick] (3) edge[bend right=30] (9);

\end{tikzpicture}
\caption{DDM labeling of $\ar{W}_4 ^{(3)}$}
\label{fig:DDMwheelwindmill_extendedtop}
\end{figure}
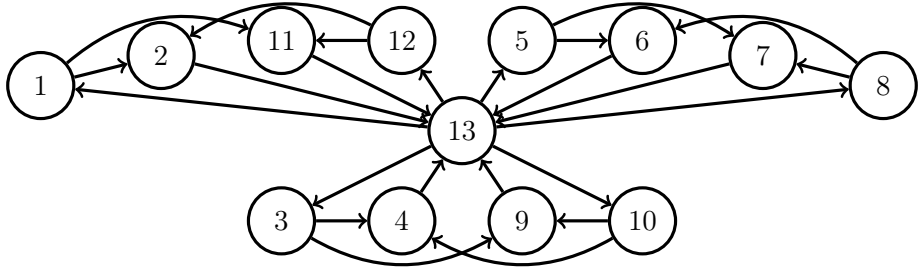

\section{Graph operations and DDMOGs}\label{operations}

We can use known DDMOGs to create new DDMOGs in various other ways. First, we discuss how to create disconnected DDMOGs using connected DDMOGs. In the final   section, we introduce a new way to create DDMOGs from smaller oriented graphs by using the labels on the graphs to create a new graph.

\subsection{A disconnected example}

The following theorem shows us that if we start with a collection of DDMOGs that are are balanced (and possibly one that is not balanced), it is fairly straightforward to construct a disconnected DDMOG that is a union of these DDMOGs using Lemma \ref{lem:shift}. However, if a union of DDMOGs happens to contain more than one component that is not balanced, it is not immediately clear if a DDM labeling will exist for this union.

 \begin{theorem}\label{thmDisconnect} Any oriented graph $\overrightarrow{G}=\bigcup\limits_{i=1}^{k} \overrightarrow{G_i}$ is DDMOG if $\overrightarrow{G_1}$ is any DDMOG and $\overrightarrow{G_2},\dots, \overrightarrow{G_k}$ are DDMOGs with $imb(\overrightarrow{G_i})=0$ for every $2 \leq i \leq k$. 
 \label{thm:discon} \end{theorem} 

Proof: Assume $\AR{G}_i$ has $n_i$ vertices and let $f_i$ be a DDM labeling associated with graph $\AR{G}_i$. Note that the number of vertices in $\AR{G}$ is $n=n_1+\dots+n_k$. Then, define the bijection $f:\AR{G} \rightarrow \{1,2, \dots, n\}$ as 
\[
f(u)=%
\begin{cases}
    f_1(u) & \text{if } u \in V(\AR{G_1})\\
    f_i(u)+(n_1+\dots+n_{i-1}) & \text{if } u \in V(\AR{G_i}), 2 \leq i \leq k. 
\end{cases}
\]
 
Every vertex in $\AR{G_1}$ still has weight 0 since the labels of $V(\AR{G_1})$ are the same in $f$ as in $f_1$. In graphs $\AR{G_2}, \dots, \AR{G_n}$ since each vertex has imbalance 0, the weight would remain zero by Lemma \ref{lem:shift}. Hence, the new labeling is also DDM. \qed

An example of Theorem \ref{thm:discon} is found in 
Figure~\ref{fig:chainWheelCone} ($W_4 \cup 2C_6(1,2))$. 
\\

\begin{figure}[H]
    \centering
     \begin{tikzpicture}  [node distance=1.4cm,base node/.style={circle,draw,minimum size=21pt}]
    
   \node[base node, very thick] (v1) at (-15,1) {1};
   \node[base node, very thick] (v2) at (-13.5,1) {2};
     \node[base node, very thick] (v3) at (-12,1) {3};
     \node[base node, very thick] (v4) at (-10.5,1) {4};
   \node[base node, very thick] (v5) at (-12.75,-0.75) {5};
     \path[->, very thick] (v1) edge  (v2); 
     \path[->, very thick] (v4) edge  (v3); 
      \path[->, very thick] (v5) edge  (v1); 
     \path[->, very thick] (v5) edge  (v4); 
      \path[->, very thick] (v2) edge  (v5); 
     \path[->, very thick] (v3) edge  (v5); 

     \path[->, very thick] (v1) edge [bend left=30] (v3);
      \path[->, very thick] (v4) edge [bend right=30] (v2);

       \node[base node, very thick] (v6) at (-8.75,4.5) {6};
   \node[base node, very thick] (v7) at (-11,3) {7};
   \node[base node, very thick] (v8) at (-9.5,3) {8};
     \node[base node, very thick] (v9) at (-8,3) {9};
     \node[base node, very thick] (v10) at (-6.5,3) {10};
   \node[base node, very thick] (v11) at (-8.75,1.75) {11};
        \path[->, very thick] (v1) edge [bend left=30] (v3);
      \path[->, very thick] (v4) edge [bend right=30] (v2);

     \path[->, very thick] (v7) edge  (v8); 
     \path[->, very thick] (v10) edge  (v9); 
      \path[->, very thick] (v11) edge  (v7); 
     \path[->, very thick] (v11) edge  (v10); 
      \path[->, very thick] (v8) edge  (v11); 
     \path[->, very thick] (v9) edge  (v11); 
        \path[->, very thick] (v7) edge [bend left=30] (v9);
      \path[->, very thick] (v10) edge [bend right=30] (v8);
 \path[->, very thick] (v6) edge [bend right=30] (v7); 
 \path[->, very thick] (v6) edge [bend left=30] (v10); 
   \path[->, very thick] (v8) edge  (v6); 
     \path[->, very thick] (v9) edge  (v6);

       \node[base node, very thick] (v12) at (-4.75,2.5) {12};
   \node[base node, very thick] (v13) at (-7,1) {13};
   \node[base node, very thick] (v14) at (-5.5,1) {14};
     \node[base node, very thick] (v15) at (-4,1) {15};
     \node[base node, very thick] (v16) at (-2.5,1) {16};
   \node[base node, very thick] (v17) at (-4.75,-0.75) {17};
 \path[->, very thick] (v12) edge [bend right=30] (v13); 
 \path[->, very thick] (v12) edge [bend left=30] (v16); 
   \path[->, very thick] (v14) edge  (v12); 
     \path[->, very thick] (v15) edge  (v12); 

     \path[->, very thick] (v13) edge  (v14); 
     \path[->, very thick] (v16) edge  (v15); 
      \path[->, very thick] (v17) edge  (v13); 
     \path[->, very thick] (v17) edge  (v16); 
      \path[->, very thick] (v14) edge  (v17); 
     \path[->, very thick] (v15) edge  (v17); 
        \path[->, very thick] (v13) edge [bend left=30] (v15);
      \path[->, very thick] (v16) edge [bend right=30] (v14);
 
      \end{tikzpicture}
        \caption{A disconnected DDMOG: $\ar{W_4} \cup 2\ar{C_6}(1,2)$} 
      
    \label{fig:chainWheelCone}
\end{figure}
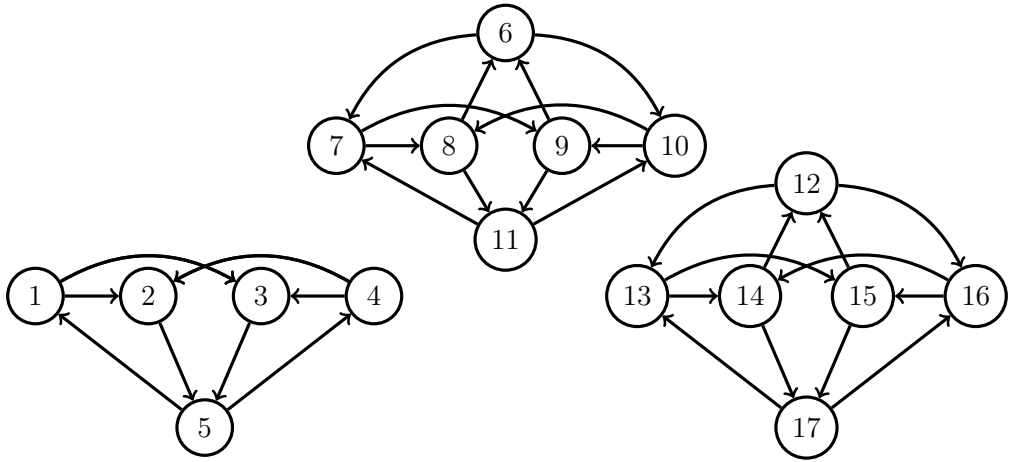

\subsection{Weighted Sums}

Here we introduce a new type of graph sum formed by taking two oriented graphs with labelings and adding additional edges based on the weights at each vertex. 
Let $V_f^i(\ar{G})=\{v \in V(\ar{G}) \, : \, wt_f(v)=i\}$ and $V_f^{-i}(\ar{G})=\{v \in V(G) \, : \,  wt_f(v)=-i\}$, for an oriented graph $\vec{G}$ with a labeling $f$.  Note that Definition \ref{def:whtsum} defines a new oriented graph, but does not define a labeling for it.

\begin{definition}\label{def:whtsum}
     
  Let $s \in \mathbb Z^+$ and let $\ar{G}$ be an oriented graph with vertex set $V(\ar{G}) = \{v_1,\dots, v_n\}$ and associated labeling function
      $g: V(\ar{G}) \rightarrow \{1,2,\dots, n\}$,
      where $g(v_i)=i$. If $\ar{H}$ is an oriented graph with labeling $h: V(H) \rightarrow \mathbb{Z}^+$, where  
       $\{ |wt_h(u)| : u \in V(\ar{H}) \} \subseteq \{0\} \cup \{ g(v)+s : v \in V(\ar{G})\}$ for some $s \in \mathbb Z$,  then the weighted sum of $\ar{G}$ and $\ar{H}$ shifted by $s$ denoted $\ar{G} \oplus_{wt_h}^s \ar{H}$
      is the oriented graph with vertex set $V(\AR{G} \oplus_{wt_h}^s \AR{H})=V(\AR{G}) \cup V(\AR{H})$ and $E(\AR{G} \oplus_{wt_h}^s \AR{H})=E(\AR{G}) \cup E(\AR{H}) \cup (\bigcup\limits_{i=1}^n (E^i \cup E^{-i}))$, where $E^i=\{v_i u:u \in V_h^{-i-s}(\AR{H})\}$ and $E^{-i}=\{uv_i:u \in V_h^{i+s}(\AR{H})\}$. 

      When the labeling $h$ and the shift $s$ are clear, we will  denote the weighted sum by $\AR{G} \oplus_{wt} \AR{H}$ instead of $\AR{G} \oplus_{wt_h}^s \AR{H}$. 
\end{definition}

As a first example of a new DDMOG constructed with the sum as defined in Definiton~\ref{def:whtsum}, we    interpret the construction given in Theorem \ref{thm:imb1} as a weighted sum. In particular, let $\overrightarrow{G}$ be a DDMOG with $imb(\overrightarrow{G})=1$ and DDM labeling $f$.  Let $h: V(\overrightarrow{G}) \rightarrow \{2, \dots, n+1\}$ be defined as $h(v)=f(v)+1$ for every $v \in V(\overrightarrow{G})$.  Note that because $imb(\AR{G})=1$, $\{|wt_h(v)| : v \in V(G)\}=\{0,1\}$. Then, $K_1 \oplus_{wt_h}^0 \overrightarrow{G}$ is a balanced DDMOG.  
The circulant can be seen as an example of this construction, $K_1 \oplus_{wt} \AR{W_4}$; see Figure~\ref{fig:bracelet1} for another example of this construction.

\begin{theorem}\label{thmWgt1} Let $\overrightarrow{G}$ be a DDMOG on $n$ vertices with DDM labeling $g$ and let $\overrightarrow{H}$ be an oriented graph on $m$ vertices with a bijective labeling $h: V(\AR{H}) \rightarrow \{n+1,\dots, n+m\}$ such that $k=\displaystyle \max_{v \in V(\AR{H})} \{|wt_h(v)|\} \leq n$ and for every $0 \leq i \leq k$, $\displaystyle \sum_{v \in V_h^i(\AR{H})} h(v)=\sum_{v \in V_h^{-i}(\AR{H})} h(v)$.  Then, $\AR{G} \oplus_{wt_h}^0 \AR{H}$ is a DDMOG on $n+m$ vertices. \end{theorem}

\begin{proof} Since $k\leq n$, we know that $\AR{G} \oplus_{wt_h}^0 \AR{H}$ is defined.  Let $f:V(\AR{G} \oplus_{wt_h}^0 \AR{H}) \rightarrow \{1, 2, \dots, n+m\}$ be defined as:
\[
f(u)=
    \begin{cases}
        g(u) & \text{if } u \in V(\AR{G})\\
        h(u)  & \text{if } u \in V(\AR{H}).
    \end{cases}
\]
We claim that $f$ is a DDM labeling of $\AR{G} \oplus_{wt_h}^0 \AR{H}$. First, note that $f$ will be a bijection with $\{1, 2, \dots, n+m\}$ as the vertices of $\AR{G}$ will have the labels $\{1,2,\dots, n\}$ and the vertices of $\AR{H}$ will have the labels $\{n+1, \dots, n+m\}$. Consider a vertex $u\in V_h^i(H)$ {for $0<i\leq k$}. In $\AR{G} \oplus_{wt_h}^0 \AR{H}$, the edge $uv$ where $v\in V(\AR{G})$ with $g(v)=i$ is present and thus, $wt_f(u)=wt_h(u)-i=i-i=0$. Similarly, if $u \in V_h^{-i}(\AR{H})$, then in $\AR{G} \oplus_{wt_h}^0 \AR{H}$ the edge $vu$ is present and thus $wt_f(u)=wt_h(u)+i=-i+i=0$. If $u \in V_h^0(H)$, then $wt_f(u)=wt_h(u)=0$ as no edges have been added to these vertices.  Likewise, if $v \in V(\AR{G})$ and $f(v)>k$, then $wt_f(v)=wt_g(v)=0$.  

Let $v \in V(\AR{G})$ with $g(v)=i$ for $1 \leq i \leq k$.
Then, 
$$\displaystyle wt_f(v)=wt_g(v)+\left(\sum_{v \in V_h^i(\AR{H})} h(v)-\sum_{v \in V_h^{-i}(\AR{H})} h(v)\right)=0+0=0.$$ 
Hence, $f$ is a DDM labeling of $\AR{G} \oplus_{wt_h}^0 \AR{H}$.
\end{proof} 
Examples of the construction presented in Theorem~\ref{thmWgt1} are seen in Figure~\ref{fig:bracelet1} and  Figure~\ref{fig:G7ornaments}.

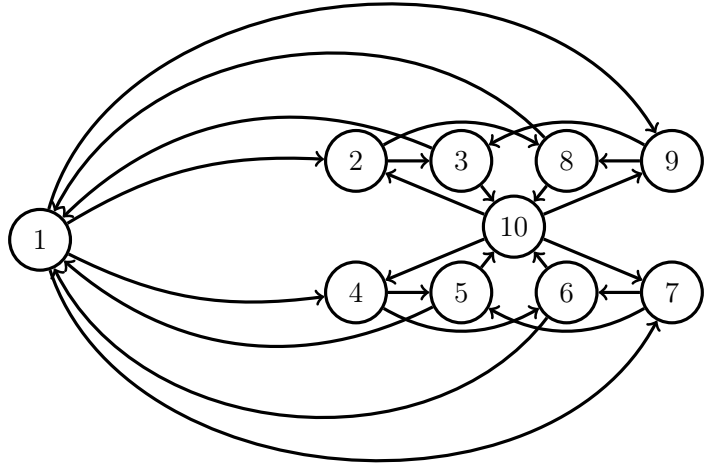
\begin{figure}[H]
\centering
\begin{tikzpicture}
[node distance=1.4cm, base node/.style={circle,draw,minimum size=23pt}, scale=0.7]

\node[base node, very thick] (2) at (-3, 1.5) {2};
 
\node[base node, very thick] (3) at (-1, 1.5) {3};
 
\node[base node, very thick] (8) at (1, 1.5) {8};
 
\node[base node, very thick] (9) at (3, 1.5) {9};

\node[base node, very thick] (10) at (0, .25) {10};

\node[base node, very thick] (1) at (-9, 0) {1};

\node[base node, very thick] (4) at (-3, -1) {4};
 
\node[base node, very thick] (5) at (-1, -1) {5};
 
\node[base node, very thick] (6) at (1, -1) {6};
 
\node[base node, very thick] (7) at (3, -1) {7};

\path[->, very thick] (10) edge (2);
\path[->, very thick] (10) edge (9);
\path[->, very thick] (10) edge (4);
\path[->, very thick] (10) edge (7);


\path[->, very thick] (2) edge (3);
\path[->, very thick] (2) edge[bend left=30] (8);
\path[->, very thick] (9) edge (8);
\path[->, very thick] (9) edge[bend right=30] (3);

\path[->, very thick] (4) edge (5);
\path[->, very thick] (4) edge[bend right=30] (6);
\path[->, very thick] (7) edge (6);
\path[->, very thick] (7) edge[bend left=30] (5);
 
 \path[->, very thick] (5) edge (10);
  \path[->, very thick] (6) edge (10);
   \path[->, very thick] (3) edge (10);
    \path[->, very thick] (8) edge (10);
 
\path[->, very thick] (1) edge[bend left=17] (2);
\path[->, very thick] (1) edge[bend right=17] (4);
\path[->, very thick] (1) edge[bend right=67] (7);
\path[->, very thick] (1) edge[bend left=67] (9);

\path[->, very thick] (3) edge[bend right=34] (1);
\path[->, very thick] (5) edge[bend left=34] (1);

\path[->, very thick] (8) edge[bend right=55] (1);
\path[->, very thick] (6) edge[bend left=58] (1);
\end{tikzpicture}
\caption{A DDM labeling of $K_1 \oplus_{wt}^0 \ar{W_4}^{(2)}$ 
}
\label{fig:bracelet1}
\end{figure}

 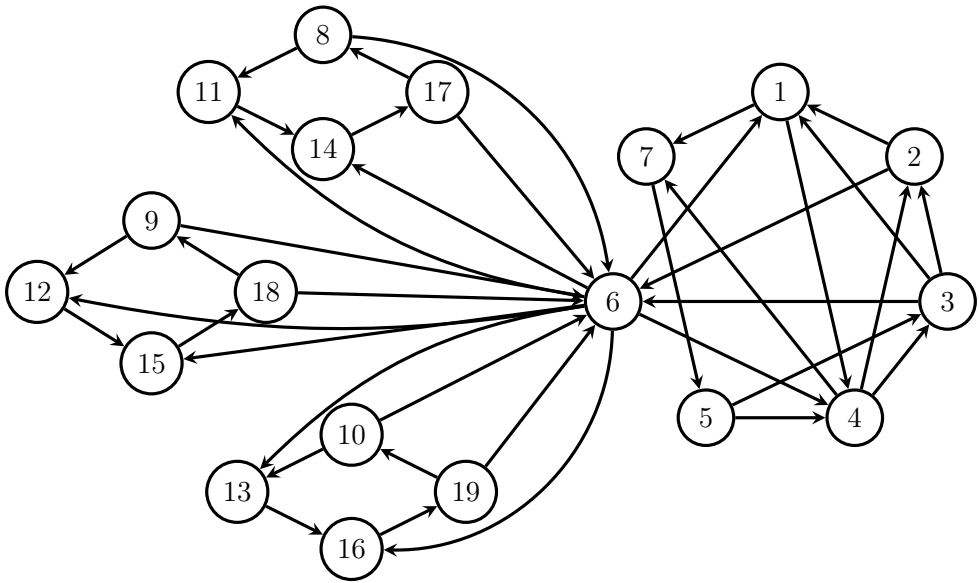
\begin{figure}[H]
        \centering
        \begin{tikzpicture}[node distance=0.5cm,base node/.style={circle,draw,minimum size=21pt}, ->, >=stealth, very thick, scale=0.76]
            \node[base node] (v1) at (90:3) {1};
            \node[base node] (v2) at ({90-360/7}:3) {2};
            \node[base node] (v3) at ({90-2*360/7}:3) {3};
            \node[base node] (v4) at ({90-3*360/7}:3) {4};
            \node[base node] (v5) at ({90-4*360/7}:3) {5};
            \node[base node] (v6) at ({90-5*360/7}:3) {6};
            \node[base node] (v7) at ({90-6*360/7}:3) {7};

              \node[base node] (8) at (-8,4) {8};
               \node[base node] (11) at (-10,3) {11};
                 \node[base node] (17) at (-6,3) {17};
              \node[base node] (14) at (-8,2) {14};
              \path (17) edge (8);
               \path (14) edge (17);
                \path (11) edge (14);
                 \path (8) edge  (11);

              \node[base node] (9) at (-11,0.75) {9};
               \node[base node] (12) at (-13,-0.5) {12};
                 \node[base node] (18) at (-9,-0.5) {18};
              \node[base node] (15) at (-11,-1.75) {15};
              \path (18) edge (9);
               \path (15) edge (18);
                \path (12) edge (15);
                 \path (9) edge (12);
              \node[base node] (10) at (-7.5,-3) {10};
               \node[base node] (13) at (-9.5,-4) {13};
                 \node[base node] (19) at (-5.5,-4) {19};
              \node[base node] (16) at (-7.5,-5) {16};
              \path (19) edge (10);
               \path (10) edge (13);
                \path (13) edge (16);
                 \path (16) edge (19);

                  \path (17) edge (v6);
                   \path (8) edge [bend left=39] (v6);
                    \path (v6) edge [bend left=18] (11);
                       \path (v6) edge (14);

                       \path (10) edge (v6);
                         \path (9) edge (v6);
                        \path (v6) edge  (15);
                         \path (v6) edge [bend left=10] (12);

                       \path (18) edge (v6);
                         \path (19) edge (v6);
                        \path (v6) edge [bend right=18] (13);
                         \path (v6) edge [bend left=45] (16);

            \path (v1) edge (v4);
            \path (v1) edge (v7);
            \path (v2) edge (v1);
            \path (v2) edge (v6);
            \path (v3) edge (v1);
            \path (v3) edge (v2);
            \path (v3) edge (v6);
            \path (v4) edge (v2);
            \path (v4) edge (v3);
            \path (v4) edge (v7);
            \path (v5) edge (v3);
            \path (v5) edge (v4);
            \path (v6) edge (v1);
            \path (v6) edge (v4);
            \path (v7) edge (v5);      
        \end{tikzpicture}        
        \caption{A DDM labeling of $\ar{R_{7}} \oplus_{{wt}}^0\ar{3C_4}$ where $\displaystyle\max_{v\in\ar{3C_4}} \{|wt_h(v)|\}= 6$ 
        } \label{fig:G7ornaments}
    \end{figure}

\vspace{3mm}

Figure~\ref{fig:G7ornaments} is constructed from a disconnected graph  with three components  and a connected DDMOG. We can generalize this to form a new DDMOG that is formed from a new DDMOG by taking the weighted sum of a DDMOG with a collection of $C_4$'s.

\begin{corollary}\label{cor:attachOrnaments}
 Let $\overrightarrow{G}$ be a DDMOG of order $n$. Then, for $1 \leq \ell \leq n/2$ there exists a labeling  $h$ on $\ell\overrightarrow{C_4}$ such that $\overrightarrow{G} \oplus_{{wt} _h}^0 \ell\overrightarrow{C_4}$ is a DDMOG.

\end{corollary}

\begin{proof}
    Let $\overrightarrow{G}$ be a DDMOG with $n\geq 5$ vertices with DDM labeling $g:V(\overrightarrow{G}) \rightarrow \{1,2, \dots, n\}$. Let $1\leq \ell \leq n/2$, and let vertex $v$ in $\overrightarrow{G}$ have label $2\ell$. 

    Let $V_i:= \{v_{i1},v_{i2},v_{i3},v_{i4}\}$ and $E_i : = \{v_{i1}v_{i2}, v_{i2}v_{i3},v_{i3}v_{i4},  v_{i4}v_{i1} \}$ be the sets of  vertices and edges of the $i$th copy of $\overrightarrow{C_4}$, respectively. 

Let $\AR{H}$ denote the collection of $l$ copies of $\AR{C_4}$. Define a labeling $h: V(\AR{H}) \rightarrow \{n+1,\dots,n+4l\}$ by 
\[ h(u)= \begin{cases}
			n+i, & \text{if $u=v_{i1}, \ 1 \leq i \leq \ell$}\\
            n+\ell+i, & \text{if $u=v_{i2}, \ 1 \leq i \leq \ell$} \\
            n+2\ell+i, & \text{if $u=v_{i3}, \ 1 \leq i \leq \ell$} \\
            n+3\ell+i, & \text{if $u=v_{i4}, \ 1 \leq i \leq \ell$.} \\
            
		 \end{cases}\]
Note that for every $u \in V(\AR{H})$, $wt_h(u)=2\ell$ or $wt_h(u)=-2\ell$.

This gives $V_h^{2\ell}(H)=\{v_{11}, v_{21}, \dots, v_{\ell 1}, v_{14}, \dots, v_{\ell 4} \}$ and $V_h^{-2\ell}(H)=\{v_{12}, v_{22}, \dots, v_{l2}, v_{13}, \dots, v_{\ell 3} \}$ and thus $$\sum_{v \in V_h^{2\ell}(H)} h(v)=\sum_{i=1}^{\ell} n+i+n+3\ell+i
=\sum_{i=1}^{\ell} (2n+3\ell+2i),$$  $$\text{and } \; \sum_{v \in V_h^{-2\ell}(H)} h(v)=\sum_{i=1}^{\ell} n+\ell+i+n+2\ell+i=\sum_{i=1}^{\ell} (2n+3\ell+2i).$$

For $f$ defined as $f(v)= g(v)$ if $v \in \AR{G}$, and $f(v)=h(v)$ if $v \in \AR{H}$, we have that  the conditions of Theorem \ref{thmWgt1} are met, and we have a DDM labeling of $\overrightarrow{G} \oplus_{{wt}_h}^0 \ell\overrightarrow{C_4}$.
\end{proof}

Up to this point, our constructions have included weighted sums of the form $\AR{G} \oplus_{wt}^s \AR{H}$ where $s=0$. The next construction shows an example where the shift $s$ is nonzero.

\begin{theorem}\label{thmGlueShift} Let $\AR{G}$ be a DDMOG on $n$ vertices with $imb(\AR{G})=0$. Let $\AR{H}$ be an oriented graph with bijective labeling $h: V(\AR{H}) \rightarrow \{1,2,\dots, m\}$ where for every vertex $v \in V(\AR{H})$ it holds that either $m+1 \leq |wt_h(v)| \leq m+n$ or $wt_h(v)=0$, and
\[\displaystyle \sum_{v \in V_h^{i+m}(\AR{H})} h(v)=\sum_{v \in V_h^{-i-m}(\AR{H})} h(v)\]
 for each $1 \leq i \leq n$.
Then, $\AR{G}\oplus_{wt_h}^m\AR{H}$ is a DDMOG on $n+m$ vertices. 
\end{theorem}

\begin{proof} 

Let $g$ be a DDM labeling on $\AR{G}$ with $n$ vertices. Let $\ar{H}$ be an oriented graph with labeling $h : V(\ar{H}) \rightarrow \{1,2, \hdots, m\}$.  
Let $f:V(\AR{G} \oplus_{wt_{h}}^m \AR{H}) \rightarrow \{1, 2, \dots, n+m\}$  be defined as: \begin{center}
$f(u)=
    \begin{cases}
        h(u) & \text{if } u \in V(\AR{H})\\
        g(u) + m & \text{if } u \in V(\AR{G}).
    \end{cases}$. \end{center}
We claim that $f$ is a DDM labeling of $\AR{G} \oplus_{wt_h}^m \AR{H}$.

 First, note that $f$ will be a bijection since the $m$ vertices of $\AR{H}$ will receive the labels $\{1,2,\dots,m\}$ and the $n$ vertices of $\AR{G}$ will receive the labels $\{m+1,m+2,\dots,m+n\}$. 
Let $g_m(u)=g(u)+m$ and recall that by assumption we have $imb(\AR{G})=0$, thus by Lemma \ref{lem:shift} it holds that $wt_{g_m}(u)=0$ for all $u \in V({\AR{G}})$.

By definition we have $E(G \oplus_{wt_h}^m H)=E(G) \cup E(H) \cup ( \bigcup\limits_{i=1}^n (E^i \cup E^{-i}))$, where  $E^{i}=\{v_i u:u \in V_h^{-i-m}(H)\}$ and $E^{-i}=\{uv_i:u \in V_h^{i+m}(H)\}$. 

Note that for $v_i \in V(\AR{G})$ with $g_m(v)=i$, where  $m+1 \leq i \leq m+n$, we have that $wt_f(v_i)=wt_{g_m}(v)+ \left( \displaystyle \sum_{u\in V_h^{i+m}(H)} f(u)- \sum_{u\in V_h^{-i-m}(H)}f(u) \right) =0$.

If $u \in V({H})$ with $wt_{h}(u)=0$, then $wt_f(u)=0$ as no new edges are added to these vertices.
  If  $u \in V_h^{i+m}(H)$, then $u \in V({H})$ with $wt_h(u)=i+m$  where  $1 \leq i \leq n$, and the edge $uv_i$ where $v_i \in V(\AR{G})$ with $f(v_i)=i+m$ is present in $\AR{G} \oplus_{wt_h}^m \AR{H}$ and thus, $wt_f(v)=wt_h(v)-(i+m)=(i+m)-(i+m)=0$.  
  Similarly, if $u\in V_h^{-i-m}(H)$ where $1 \leq i \leq n$, then in $\AR{G} \oplus_{wt_h}^m \AR{H}$ the edge  $v_iu$ is present with $f(v_i)=i+m$ and thus $wt_f(u)=wt_h(u)+i+m=(-i-m)+(i+m)=0$. Since every vertex in $\AR{G} \oplus_{wt_h}^m \AR{H}$ has weight zero, we have a DDM labeling of $\AR{G} \oplus_{wt_h}^m \AR{H}$.

  \end{proof}

Figure~\ref{fig:example3_final} gives us an example of the construction presented in Theorem~\ref{thmGlueShift}.

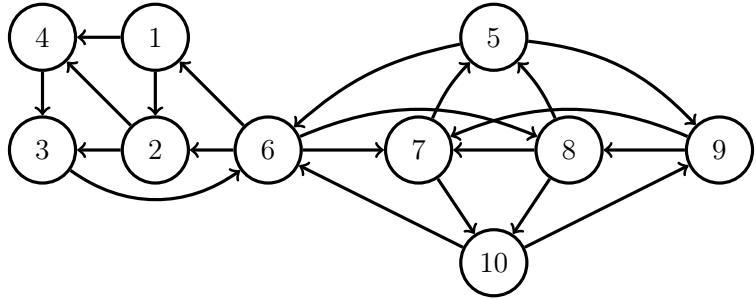
\begin{figure}[H]
\centering
\begin{tikzpicture}
[node distance=1.4cm, base node/.style={circle,draw,minimum size=25pt}]

\node[base node, very thick] (5) at (1, 4) {5};
\node[base node, very thick] (6) at (-2, 2.5) {6};
\node[base node, very thick] (7) at (0, 2.5) {7};
\node[base node, very thick] (8) at (2, 2.5) {8};
\node[base node, very thick] (9) at (4, 2.5) {9};

\node[base node, very thick] (1) at (-3.5, 4) {1};
\node[base node, very thick] (2) at (-3.5, 2.5) {2};
\node[base node, very thick] (3) at (-5, 2.5) {3};
\node[base node, very thick] (4) at (-5, 4) {4};

\node[base node, very thick] (10) at (1, 1) {10};

\path[->, very thick] (5) edge [bend right=15] (6);
\path[->, very thick] (7) edge [bend left=10] (5);
\path[->, very thick] (8) edge [bend right=10] (5);

\path[->, very thick] (6) edge (7);
\path[->, very thick] (6) edge[bend left=23] (8); 
\path[->, very thick] (8) edge (7);
\path[->, very thick] (9) edge[bend right=23] (7); 
\path[->, very thick] (9) edge (8);
\path[->, very thick] (5) edge[bend left=20] (9);  

\path[->, very thick] (6) edge (1);
\path[->, very thick] (6) edge (2);
\path[->, very thick] (3) edge[bend right=36] (6); 

\path[->, very thick] (1) edge (2);
\path[->, very thick] (1) edge (4);
\path[->, very thick] (2) edge (3);
\path[->, very thick] (2) edge (4);
\path[->, very thick] (4) edge (3);

\path[->, very thick] (7) edge (10);
\path[->, very thick] (8) edge (10);
\path[->, very thick] (10) edge (9);
\path[->, very thick] (10) edge (6);


\end{tikzpicture}
\caption{DDMOG labeling of $\ar{C_4} \oplus_{wt}^4 \ar{C_6}(1,2)$%
}
\label{fig:example3_final}
\end{figure}

\section{Conclusion}\label{sec13}

In this paper, we expanded on what is known about DDM labelings on oriented graphs.  We found that DDM labelings of DDMOGs exhibit interesting properties  (Lemma \ref{lem:balancelabelsum}) and that linear algebra is fruitful in proving these properties.  We demonstrated that graph families, such as windmills of wheels, can be investigated to determine if they are DDMOG or not.  We also showed that there are a variety of ways to construct new DDMOGs from previously existing ones.  One theme that emerged is that shifting DDM labelings of oriented graphs that have imbalance 0 is particularly helpful during these constructions.  Furthermore, previously known operations like vertex coalescence are also helpful when growing new DDMOGs from old ones.  So is combining graphs in familiar yet distinct ways, as exemplified by the  weighted sum in Definition \ref{def:whtsum}. 

With such a variety of results, there is a rich mix of classical and novel pathways forward for exploring DDMOGs.  As one such pathway, recall that DDMOGs exist for arbitrarily large order $n\geq 5$.  On the other hand, not every graph is DDMO.  The reader can check for example that no orientation of the cycle $C_n$ is DDMO.  This begs the question:  Which graphs are DDMO and which are not? Can we find properties which distinguish DDMOGs and non-DDMOGs? Furthermore, although most of the DDMOGs in this paper have been connected, disconnected DDMOGs are of interest as well.  For example, is it possible for a union of non-DDMOGs to be DDMO?

Another particularly intriguing direction forward (pun intended) involves studying DDM labelings of directed graphs (digraphs).  Because an oriented graph is a digraph, each oriented graph result in this paper extends to digraphs, as does the terminology.  For example, we call a digraph $D$ a \textit{difference distance magic directed graph (DDMDG)} if $D$ exhibits a DDM labeling.  Despite similarities that exist between DDMOGs and DDMDGs,  there are innate differences too.  For example, whereas there exist graphs like cycles that are not DDMO, every graph emits at least one direction that yields a DDMGD.  To see this, make every edge bidirectional and note that every labeling then becomes a DDM labeling.  Hence, one open area of exploration involves understanding additional differences between DDMOGs and DDMDGs. 








\bibliographystyle{abbrv}
\bibliography{magic}

\end{document}